\documentclass[11pt]{article}

\usepackage[section]{colestyle}

\newcommand{\shiftzero}{\al_0}
\newcommand{\shiftone}{\al_1}

\title{Precise asymptotics for Fisher-KPP fronts}

\author{Cole Graham}

\begin{document}

\maketitle

\begin{abstract}
  We consider the one-dimensional Fisher-KPP equation with step-like initial data.
  Nolen, Roquejoffre, and Ryzhik showed in \cite{NRR2} that the solution $u$ converges at long time to a traveling wave $\phi$ at a position $\ti\sigma(t) = 2t - (3/2)\log t + \shiftzero - 3\sqrt{\pi}/\sqrt{t}$, with error $\m O(t^{\gamma-1})$ for any $\gamma>0$.
  With their methods, we find a refined shift $\sigma(t) = \ti \sigma(t) + \mu_* (\log t)/t + \shiftone/t$ such that in the frame moving with $\sigma$, the solution $u$ satisfies $u(t,x) = \phi (x) + \psi(x)/t + \m O(t^{\gamma-3/2})$ for a certain profile $\psi$ independent of initial data.
  The coefficient $\shiftone$ depends on initial data, but $\mu_* = 9(5-6\log 2)/8$ is universal, and agrees with a finding of Berestycki, Brunet, and Derrida \cite{BBD} in a closely-related problem.
  Furthermore, we predict the asymptotic forms of $\sigma$ and $u$ to arbitrarily high order.
\end{abstract}


\section{Introduction}
\label{sec:intro}

We study solutions to the Fisher-KPP equation
\begin{equation}
  \label{eq:FKPP}
  u_t = u_{xx} + u(1-u)\quad \text{with }(t,x)\in \R_+\times \R.
\end{equation}
For initial data we take $u(0,x)=u_0(x)$ for $x\in \R$, where $u_0$ is a compact perturbation of a step function.
That is, there exists $L\geq 0$ such that $u_0(x)=1$ when $x\leq -L$ and $u_0(x) = 0$ when $x\geq L$.
We further assume that $0\leq u_0\leq 1$ on $\R$, so that $0<u<1$ on $\R_+\times \R$.
Our results will hold under weaker hypotheses on $u_0$, but we do not explore this issue in the present work.
We study the long-time asymptotics of $u$.

This question has a rich history, beginning with Fisher's introduction of equation \eqref{eq:FKPP} in \cite{Fisher}.
Fisher studied \emph{traveling front} solutions to \eqref{eq:FKPP}, which have the form $u(t,x) = \phi_c(x-ct)$, where $\phi_c\colon \R\to (0,1)$ satisfies
\begin{equation}
  \label{eq:phi_c}
  -c\phi_c' = \phi_c'' + \phi_c - \phi_c^2,\quad \phi_c(-\infty) = 1,\quad \phi_c(+\infty)=0.
\end{equation}
Such solutions model steady-speed invasions of the unstable state $0$ by the stable state 1.
Fisher used heuristic and numerical arguments to identify the minimal speed $c_* = 2$ of traveling fronts.
At the minimal speed there exists a front $\phi_{c_*}$ unique up to translation.
We use the translation $\phi$ satisfying
\begin{equation}
  \label{eq:phi_asymp}
  \phi(s) = (s+k)e^{-s} + \m O(e^{-(1+\omega)s})\quad \text{as}\;\;s\to+\infty
\end{equation}
for universal constants $k\in \R$ and $\omega>0$.

In the same year as \cite{Fisher}, Kolmogorov, Piskunov, and Petrovsky published their groundbreaking work \cite{KPP}.
The authors show that if $u_0$ is a step function, the solution $u$ converges to the minimal-speed front $\phi$, in the sense that
\begin{equation}
  \label{eq:sigma_limit}
 \lim_{t\to\infty} u(t,x+\sigma(t)) = \phi(x)
\end{equation}
uniformly on compact sets in $x$, for some function $\sigma$ satisfying
\begin{equation}
  \sigma(t) = 2t + \smallO(t)\quad \text{as }t\to+\infty.
\end{equation}
The precise nature of this convergence has since been well-studied, and is the subject of this work.

In a striking series of papers \cite{Bramson78,Bramson83}, Bramson proved that $\sigma$ is \emph{not} asymptotically constant.
Rather:
\begin{theorem}[Bramson]
  \label{thm:Bramson}
  There exists $\shiftzero\in \R$ such that
  \begin{equation}
    \label{eq:Bramson}
    \sigma(t) = 2t - \frac 3 2 \log t + \shiftzero + \smallO(1)\quad \text{as }t\to+\infty.
  \end{equation}
\end{theorem}
In fact, Bramson established the same result for a precisely-determined class of initial data that decay rapidly as $x\to+\infty$.
Significantly, the constant shift $\shiftzero$ depends on the initial data, but the coefficient of the logarithmic delay does not.
In this sense the logarithmic term is ``universal.''
Bramson used elaborate probabilistic methods to prove Theorem \ref{thm:Bramson}, drawing on intimate connections between the FKPP equation \eqref{eq:FKPP} and the stochastic process of branching Brownian motion.
Soon after, Lau \cite{Lau} provided a different proof of the results of \cite{Bramson78,Bramson83} for more general nonlinearities, using the intersection properties of solutions to parabolic Cauchy problems.

Recent years have seen substantial progress through purely PDE methods.
In \cite{HNRR}, Hamel, Nolen, Roquejoffre, and Ryzhik related the Cauchy problem for \eqref{eq:FKPP} to a moving linear Dirichlet boundary problem, and established
\begin{equation}
  \sigma(t) = 2t - \frac 3 2 \log t + \m O(1).
\end{equation}
In a subsequent work \cite{NRR1}, Nolen, Roquejoffre, and Ryzhik used the same approach to recover \eqref{eq:Bramson} for initial data of the form studied here: compact perturbations of a step function.

To further analyze $\sigma$, we must consider a slightly different question.
After all, any $o(1)$ change to $\sigma$ will still satisfy the limit \eqref{eq:sigma_limit} found by KPP.
We are therefore interested in the \emph{rate} of convergence in \eqref{eq:sigma_limit}.
That is, we wish to find further terms in $\sigma$ such that $u(t,x+\sigma(t))$ converges rapidly to $\phi(x)$.
In \cite{EvS}, Ebert and van Saarloos performed formal calculations suggesting:
\begin{equation}
  \sigma(t) = 2t - \frac 3 2 \log t + \shiftzero - \frac{3\sqrt\pi}{\sqrt t} + \smallO\left(\frac 1 {\sqrt{t}}\right).
\end{equation}
That is, \cite{EvS} predicts that for such $\sigma$,
\begin{equation}
  \label{eq:EvS_predict}
  u(t,x+\sigma(t)) = \phi(x) + \smallO(t^{-1/2})\quad\text{as}\;\;t\to\infty
\end{equation}
uniformly on compacts in $x$.
Notably, the coefficient of the $t^{-\frac 1 2}$ correction is again universal, in that it is independent of the initial data.
This is particularly striking given that a larger term, $\shiftzero$, \emph{does} depend on $u_0$.

In \cite{NRR2}, Nolen, Roquejoffre, and Ryzhik proved the $t^{-1/2}$ refinement derived by Ebert and van Saarloos.
Precisely, the authors construct an \emph{approximate solution} $\ti{u}_{\op{app}}$ incorporating both the traveling wave $\phi$ and the linear behavior of the``pulled front'' at $x\gg 2t$.
Let
\begin{equation}
  \ti\sigma(t)\coloneqq 2t - \frac 3 2 \log t + \shiftzero - \frac{3\sqrt{\pi}}{\sqrt{t}}
\end{equation}
denote their front-shift.
Then
\begin{theorem}[Nolen,~Roquejoffre,~Ryzhik]
  \label{thm:NRR2}
  There exists $\shiftzero\in \R$ depending on the initial data $u_0$ such that for any $\gamma>0$ there exists $C_\gamma>0$ also depending on $u_0$ such that
  \begin{equation}
    \abs{u(t,x+\ti\sigma(t)) - \ti{u}_{\op{app}}(t,x+\ti\sigma(t))} \leq \frac{C_\gamma(1+\abs{x})e^{-x}}{t^{1-\gamma}}\quad\text{for all }(t,x)\in [1,\infty)\times \R.
  \end{equation}
\end{theorem}
The approximate solution satisfies $\ti{u}_{\op{app}}(t,x+\ti\sigma(t)) = \phi(x) + \m O(t^{\gamma-1})$ as $t\to+\infty$ locally uniformly in $x$.
Hence Theorem \ref{thm:NRR2} proves \eqref{eq:EvS_predict}.
In \cite{Henderson}, Henderson established the same $t^{-1/2}$ correction for a related moving-boundary problem.

In a recent work \cite{BBD}, Berestycki, Brunet, and Derrida discovered a remarkable formula relating initial data and front-position in a free-boundary problem closely related to \eqref{eq:FKPP}.
Their formula predicts a universal $\frac{\log t}{t}$ correction of the form:
\begin{equation}
  \label{eq:BBD_predict}
  \tag{1.6.a}
  \sigma(t) = 2t - \frac 3 2 \log t + \shiftzero - \frac{3\sqrt\pi}{\sqrt t} + \frac 9 8(5-6\log 2)\frac{\log t}{t} + \m O\left(\frac 1 t\right).
\end{equation}
For concision, we let $\mu_* = \frac 9 8(5-6\log 2)$ denote this universal coefficient.
In the present work, we prove \eqref{eq:BBD_predict}.
Furthermore, we characterize $u$ to order $t^{-1}$, and find that it \emph{cannot} be represented as a simple shift of the traveling front $\phi$.

Our main theorem makes these observations precise.
For our front-shift, we include the $\frac{\log t}{t}$ correction predicted in \cite{BBD} and an undetermined order $\frac 1 t$ term:
\begin{equation}
  \label{eq:our_shift}
  \tag{1.6.b}
  \sigma(t) \coloneqq 2t + \shiftzero - \frac 3 2 \log t - \frac{3\sqrt{\pi}}{\sqrt{t}} + \mu_*\frac{\log t}{t} + \frac{\shiftone}{t}.
\end{equation}
The constants $\shiftzero$ and $\shiftone$ will depend on the initial data $u_0$.
There is a second correction at order $\frac 1 t$, however.
For any $\gamma>0$, we construct an approximate solution $u_{\op{app}}$ satisfying
\begin{equation}
  u_{\op{app}}(t,x+\sigma(t)) = \phi(x) + \frac 1 t \psi(x) + \m O\left(t^{\gamma-\frac 3 2}\right)\quad\text{as}\;\;t\to\infty
\end{equation}
locally uniformly in $x$.
The profile $\psi$ solves
\begin{equation}
  \label{eq:psi}
  \psi'' + 2\psi' + (1-2e^x\phi)\psi = \frac 3 2 \phi',
\end{equation}
and is independent of $u_0$.
This $\frac{\psi}{t}$ term is an effect of the $\frac 3 2 \log t$ delay in the front position.

We will show:
\begin{theorem}
  \label{thm:main}
  There exist $\shiftzero$ and $\shiftone$ in $\R$ depending on the initial data $u_0$ such that the following holds.
  For any $\gamma>0$, there exists $C_\gamma>0$ also depending on $u_0$ such that for all $(t,x)\in [3,\infty)\times \R,$
  \begin{equation}
    \label{eq:main}
    \abs{u(t,x+\sigma(t)) - u_{\op{app}}(t,x+\sigma(t))} \leq \frac{C_\gamma(1+\abs x)e^{-x}}{t^{\frac 3 2-\gamma}},
  \end{equation}
  with $\sigma$ defined in \eqref{eq:our_shift}.
\end{theorem}

\begin{remark}
  Because $\shiftone$ depends on $u_0$, we find that the asymptotic behavior of $u$ at order $\frac 1 t$ is not universal.
\end{remark}

\begin{remark}
  The $\frac 1 t$ correction $\psi$ varies in space, so from this order the asymptotics of $u$ cannot be described as simple shifts of the traveling front $\phi$.
\end{remark}

Our main theorem implies:
\begin{corollary}
  \label{cor:shift}
  For each $s\in (0,1)$, let $\sigma_s(t)\coloneqq \max \{ x\in \R ; \; u(t,x)=s \}$ denote the leading edge of $u$ at value $s$.
  Then
  \begin{equation}
    \label{eq:cor}
    \sigma_s(t) = 2t - \frac 3 2 \log t + \shiftzero + \phi^{-1}(s) - \frac{3\sqrt\pi}{\sqrt t} + \mu_*\frac{\log t}{t} + \m O\left(\frac 1 t\right).
  \end{equation}
\end{corollary}
The proofs of these results closely follow the methods of Nolen, Roquejoffre, and Ryzhik in \cite{NRR1,NRR2}.

Theorem \ref{thm:main} raises the question of the general behavior of $\sigma$ and $u$.
We informally argue the existence of a shift
\begin{equation}
  \label{eq:shift_complete_intro}
  \h\sigma(t) \sim 2t - \frac 3 2 \log t + \sum_{\substack{a\in \frac 1 2 \Z,\\ a\geq 0}}\sum_{\substack{b\in \Z,\\ 0\leq b\leq a}} \sigma_{a,b}\,t^{-a}\log^b t, 
\end{equation}
such that
\begin{equation}
  \label{eq:u_complete_intro}
  u(t,x + \h\sigma(t)) \sim \phi(x) + \sum_{\substack{a\in \frac 1 2 \Z,\\ a\geq 1}}\sum_{\substack{b\in \Z,\\ 0\leq b\leq a-1}} t^{-a}\log^b t\; u_{a,b}(x).
\end{equation}
Furthermore, for any fixed value of $a$, the corresponding terms in $u$ and $\sigma$ with maximal degree in $\log t$ are independent of $u_0$.
In this sense, ``leading logarithmic'' terms are universal.

Our paper is structured as follows.
We outline the proof of Theorem \ref{thm:main} in Section \ref{sec:outline}, and intuitively motivate the result and methods.
In Section \ref{sec:approx}, we perform the matched asymptotic expansion for $u_{\op{app}}$, and derive an implicit equation for the coefficient $\mu_*$.
In Section \ref{sec:mu}, we explicitly compute $\mu_*$, to show agreement with \cite{BBD}.
We extend our asymptotic analysis to all orders in Section \ref{sec:complete}, and thereby describe the KPP front shift to arbitrarily high order.
In Section \ref{sec:proof}, we use the approach of \cite{NRR2} to prove Theorem \ref{thm:main}.
We close with an appendix detailing an ODE lemma required in the construction of $u_{\op{app}}$.

\section*{Acknowledgements}

This work was supported by the Fannie and John Hertz Foundation and by NSF grant DGE-1656518.
We warmly thank Professor Lenya Ryzhik for introducing us to the problem, and for his constant encouragement and guidance.

\section{Proof outline}
\label{sec:outline}

Recall our main equation, which we begin from $t=1$ for convenience.
\begin{equation}
  \begin{cases}
    u_t = u_{xx} + u - u^2, & (t,x)\in (1,\infty)\times \R,\\
    u(1,\cdot) = u_0, & x\in \R.
  \end{cases}
\end{equation}
As in the introduction, we assume that the initial data $0\leq u_0\leq 1$ is a compact perturbation of a step function.
We then expect $u$ to converge to a traveling front at position $\sigma$ of the form
\begin{equation}
  \sigma(t) = 2t - \frac 3 2 \log t + \shiftzero - \frac{3\sqrt\pi}{\sqrt t} + \mu \frac{\log t}{t} + \frac{\shiftone}{t}.
\end{equation}
It is therefore natural to change coordinates to the moving frame given by
\begin{equation}
  x_{\op{new}} = x_{\op{old}} - \sigma(t).
\end{equation}

Now, $u$ is a ``pulled-front,'' meaning its dynamics are determined by its behavior along the leading tail $x\gg 1$.
In this regime, $u$ is very close to $\phi$, which decays exponentially as $x\to\infty$.
To detect detailed behavior in the tail, it is helpful to remove this exponential decay.
With this motivation, we study
\begin{equation}
  v(t,x)\coloneqq e^xu(t,x).
\end{equation}

Incorporating the shift and the exponential multiplier, \eqref{eq:FKPP} becomes
\begin{equation}
  \label{eq:transformed}
  v_t - v_{xx} - \left(\frac{3}{2t}-\frac{3\sqrt\pi}{2t^{\frac 3 2}}+\mu\frac{\log t}{t^2}+\frac{\shiftone-\mu}{t^2}\right)(v-v_x) + e^{-x}v^2 = 0\quad \text{on } (1,\infty)\times \R.
\end{equation}
In particular, when $t$ and $x$ are large, \eqref{eq:transformed} resembles the heat equation for $v$.
We thus expect the dynamics of \eqref{eq:transformed} to be driven at the diffusive scale $x\sim \sqrt t$.

With this scale in mind, we introduce the self-similar variables
\begin{equation}
  \tau\coloneqq \log t,\quad \eta\coloneqq \frac{x}{\sqrt{t}}.
\end{equation}
In these variables, \eqref{eq:transformed} becomes
\begin{equation}
  \label{eq:ss0}
  v_\tau - v_{\eta\eta} - \frac{\eta}{2}v_\eta + \left(\frac 3 2 - \frac{3\sqrt\pi}{2}e^{-\tau/2} + \mu \tau e^{-\tau} + (\shiftone-\mu)e^{-\tau}\right)\left(e^{-\tau/2}v_\eta - v\right) + e^\tau e^{-\eta e^{\tau/2}}v^2 = 0.
\end{equation}
Crucially, at any fixed $\eta>0$, the prefactor $e^\tau e^{-\eta e^{\tau/2}}$ of the nonlinear term decays rapidly.
Thus the nonlinear nature of the problem only manifests in a boundary layer near $\eta = 0$.
Furthermore, since $u\leq 1$, we have $v\leq e^{\eta e^{\tau/2}}$.
Thus when $\eta<0$, $v$ approaches 0 rapidly.
We therefore expect $v$ to approximately solve a linear Dirichlet boundary value problem on $\R_+$.

To make these heuristics precise, we construct an approximate solution $V_{\op{app}}$ through a matched pair of asymptotic expansions.
When $x\sim 1$, we solve the nonlinear equation \eqref{eq:transformed} by expanding in successively smaller orders of $t$.
For $x\sim \sqrt t$, we solve the linear part of \eqref{eq:ss0} on $\R_+$ with Dirichlet boundary data, again expanding in orders of $t = e^{\tau}$.
To link the inner expansion at $x\sim 1$ with the outer expansion at $x\sim t^{\frac 1 2}$, we match them at an intermediate scale $x = t^\eps$.
In this matching, the inner expansion $V^-$ sets additional boundary conditions on the outer expansion $V^+$, through the Neumann data $\partial_\eta V^+|_{\eta=0}$.
To solve the resulting over-determined boundary problem, we use degrees of freedom in the shift $\sigma$.
The universal coefficients of $\sigma$ are uniquely chosen to admit a solution $V^+$ satisfying the boundary conditions prescribed by $V^-$.

This method determines the universal terms $-\frac 3 2 \log t$, $-\frac{3\sqrt\pi}{\sqrt t}$, and $\mu_*\frac{\log t}{t}$.
However, it does not determine the terms $\shiftzero$ and $\frac{\shiftone}{t}$, which depend on the initial data $v_0$.
In general, the spectral properties of the Dirichlet problem make the matched expansion insensitive to shift terms of order $t^{-a}$ with $a\in \Z_{\geq 0}$.
Rather, these terms are chosen to eliminate components of the difference $v-V_{\op{app}}$.

For instance, the principal eigenfunction of the Dirichlet problem on $\R_+$ is $\eta e^{-\eta^2/4}$.
As a consequence, the leading term of $V^+$ will be $e^{\tau/2}\,\eta e^{-\eta^2/4}$ on $\R_+$.
On the other hand, \cite{NRR1} shows the existence of $q_0\in \R$ such that $v(\tau,\eta)\sim q_0 e^{\tau/2}\, \eta e^{-\eta^2/4}$ when $\tau\gg 1$.
By adjusting $\shiftzero$, we can force $q_0 = 1$, so that $v$ and $V^+$ agree to leading order.
In other words, we choose $\shiftzero$ to eliminate the principal component of $v-V_{\op{app}}$.
Similarly, $\shiftone$ will be chosen to kill the component of $v-V_{\op{app}}$ corresponding to the second eigenfunction of the Dirichlet problem.

In summary, we wish to construct an approximate solution $V_{\op{app}}$ to \eqref{eq:transformed} which closely models the exact solution $v$.
To do so, we perform a matched asymptotic expansion at the scales $x\sim 1$ and $x\sim \sqrt t$.
The universal terms of $\sigma$ are uniquely chosen to ensure the existence of such an expansion.
The remaining terms $\shiftzero$ and $\frac{\shiftone}{t}$ are then chosen so that $V_{\op{app}}$ and $v$ agree up to a certain order in the eigenbasis of the linear Dirichlet problem.
In all these steps, we closely follow \cite{NRR2}, which developed this method to the first order.

\section{Matched asymptotics for the approximate solution}
\label{sec:approx}


As described above, we transform \eqref{eq:FKPP} by translating to a moving frame and removing the exponential decay of $u$:
\begin{equation}
  x\mapsto x - 2t + \frac 3 2 \log t - \shiftzero + \frac{3\sqrt \pi}{\sqrt t} - \mu\frac{\log t}{t} - \frac{\shiftone}{t},\quad\quad v(t,x) = e^xu(t,x).
\end{equation}
Here we use an undetermined coefficient $\mu\in \R$ for the $\frac{\log t}{t}$ term in the shift.
We will show that only the special value $\mu=\mu_*$ will allow us to approximate $u$ with $\smallO\left(\frac{\log t}{t}\right)$ accuracy.

We now construct asymptotic solutions to \eqref{eq:transformed} at the scales $x\sim 1$ and $x\sim \sqrt t$.
We denote these expansions by $V^-$ and $V^+$ respectively, and match them at the intermediate position $x =  t^\eps$ to construct $V_{\op{app}}$.
Our choice of $0<\eps\ll 1$ will depend on the parameter $\gamma$ in Theorem \ref{thm:main}.

\subsection{The inner approximation}

We first take $x \sim 1$, and expand \eqref{eq:transformed} in orders of $t$.
Since we expect $\m O(t^{\gamma-3/2})$ error in Theorem \ref{thm:main}, we may discard terms of this order and smaller.
Two terms in \eqref{eq:transformed} remain, of order 1 and $t^{-1}$.
We thus use the ansatz
\begin{equation}
  V^-(t,x) = V_0^-(x) + t^{-1}V_1^-(x).
\end{equation}
Considering only order 1 terms, we find the equation for $V_0^-$:
\begin{equation}
  -(V_0^-)'' + e^{-x}(V_0^-)^2 = 0.
\end{equation}
The traveling front $\phi$ provides a natural solution:
\begin{equation}
  V_0^-(x) = e^x\phi(x).
\end{equation}
By \eqref{eq:phi_asymp} and the standard theory of traveling fronts, $V_0^-$ satisfies
\begin{equation}
  \label{eq:V0-behavior+}
  V_0^-(x) = x + k + \m O(e^{-\omega x})\quad\text{and}\quad (V_0^-)'(x) = 1 + \m O(e^{-\omega x})\quad \text{ as }x\to+\infty
\end{equation}
for some $k\in \R$, $\omega\in (0,1)$.
For convenience, we now shift the $x$-coordinate so that $k=0$.
In the other direction,
\begin{equation}
  \label{eq:V0-behavior-}
  V_0^-(x) = e^{x} + \m O(e^{(1+\omega) x})\quad\text{and}\quad(V_0^-)'(x) = e^x + \m O(e^{(1+\omega) x})\quad \text{ as }x\to-\infty.
\end{equation}

We now collect the terms of order $t^{-1}$ in \eqref{eq:transformed}:
\begin{equation}
  \label{eq:V1-}
  -(V_1^-)'' + 2e^{-x}V_0^-V_1^- = \frac 3 2 [V_0^- - (V_0^-)'].
\end{equation}
From the asymptotics of $V_0^-$, \eqref{eq:V1-} is an exponentially-small perturbation of $-(V_1^-)'' = \frac 3 2 (x-1)$ on $\R_+$.
We therefore expect
\begin{equation}
  \label{eq:V1-behavior+}
  V_1^-(x) = -\frac 1 4 x^3 + \frac 3 4 x^2 + C_1^-x + C_0^- + \m O(e^{-\omega x/2})\quad \text{as}\;\; x\to\infty,
\end{equation}
for some $C_1^-,C_0^-\in \R$.

To uniquely specify $V_1^-$, we must impose boundary conditions.
One condition is straightforward: $V_1^-$ must be a perturbation of $V_0^-$, so it must decay as $x\to -\infty$.
Furthermore, we shall find that an accurate matching between the inner and outer approximations requires $C_0^- = 0$ in \eqref{eq:V1-behavior+}.
In the appendix, we prove:
\begin{lemma}
  \label{lem:ODE}
  There exist $C_1^-\in \R$ and a solution $V_1^-$ to \eqref{eq:V1-} satisfying
  \begin{equation}
    \begin{split}
      V_1^-(x) &= -\frac 1 4 x^3 + \frac 3 4 x^2 + C_1^-x + \m O(e^{-\omega x/2}),\\
      (V_1^-)'(x) &= -\frac 3 4 x^2 + \frac 3 2 x + C_1^- + \m O(e^{-\omega x/2})
    \end{split}
  \end{equation}
  as $x\to+\infty$ and $V_1^-,(V_1^-)' = \m O (e^x)$ as $x\to-\infty.$
\end{lemma}
\noindent For the remainder of the paper, $V_1^-$ denotes this solution.

Finally, we note that $V^-$ will be spatially shifted by a time-dependent quantity $\zeta(t)$ to ensure the continuity of $V_{\op{app}}$ at the matching point $x = t^\eps$.
We defer this technicality to Section \ref{sec:proof}.

\subsection{The outer approximation}

The outer layer $V^+$ requires a more elaborate analysis, and involves several more terms.
To emphasize the diffusive nature of the problem, we change to the self-similar variables
\begin{equation}
  \tau\coloneqq \log t,\quad \eta\coloneqq \frac{x}{\sqrt t}.
\end{equation}
Recall that in these variables, $v$ satisfies \eqref{eq:ss0}.
As noted in Section \ref{sec:outline}, we will neglect the nonlinear term $e^\tau e^{-\eta e^{\tau/2}}v^2$ on $\R_+$.
Furthermore, $v$ decays rapidly on $\R_-$, so we approximate \eqref{eq:ss0} with the linear Dirichlet problem
\begin{equation}
  \label{eq:ss}
  V_\tau - V_{\eta\eta} - \frac{\eta}{2}V_\eta - \left(\frac 3 2 - \frac{3\sqrt\pi}{2}e^{-\tau/2} + \mu \tau e^{-\tau} + (\shiftone-\mu)e^{-\tau}\right)\left(V - e^{-\tau/2}V_\eta\right) = 0
\end{equation}
with $V(0,\tau)=0$ for all $\tau\geq 0$.

Consider $V^+$ near $\eta = 0$, where $V^+(\tau,\eta)\sim \partial_\eta V^+(\tau,0)\,\eta$.
We will match this behavior with $V^-(x)\sim x = e^{\tau/2}\eta$.
We therefore anticipate $\partial_\eta V^+(\tau,0)\sim e^{\tau/2}$.

This motivates our asymptotics for $V^+$: we expand in orders of $\tau$, and assume the leading order is $e^{\tau/2}$.
At fixed $x,$ we are only interested in behavior of order $t^{-1}$ or larger.
Since $V^+$ satisfies the Dirichlet condition, this corresponds to terms of order $e^{-\tau/2}$ in $V^+$.
We therefore neglect all smaller terms in \eqref{eq:ss}.
Performing this expansion, we find:
\begin{equation}
  \label{eq:ansatz}
  V^+(\tau,\eta) = e^{\tau/2}V_0^+(\eta) + V_1^+(\eta) + \tau e^{-\tau/2}V_2^+(\eta) + e^{-\tau/2}V_3^+(\eta).
\end{equation}
We impose the boundary conditions independently on each term, so $V_i^+(0)=V_i^+(\infty)=0$ for all $i=0,\ldots,3$.
By considering \eqref{eq:ss} at each successive order in $\tau$, we will obtain equations for each $V_i^+$.
Most free constants appearing in the solutions to these equations will be determined by the matching with $V^-$ at $x = t^\eps$.

Before writing the equations for $V_i^+$, we introduce
\begin{equation}
  \m L\coloneqq -\partial_\eta^2 - \frac{\eta}{2}\partial_\eta - 1,
\end{equation}
a differential operator closely connected with the left-hand side of \eqref{eq:ss}.
We are interested in the Dirichlet problem for $\m L$ on the half-line $[0,\infty)$.
The discrete spectrum of $\m L$ is $\Z_{\geq 0}$ without multiplicity.
The functions defined by
\begin{equation}
  \phi_0(\eta) \coloneqq \eta e^{-\eta^2/4},\quad \phi_{k+1} \coloneqq \phi_k''\text{ for }k\in \Z_{\geq 0}
\end{equation}
are eigenfunctions of $\m L$ satisfying $\m L\phi_k=k\phi_k$ for all $k\in \Z_{\geq 0}$.
The adjoint operator is given by
\begin{equation}
  \m L^* = -\partial_\eta^2 +\frac{\eta}{2}\partial_\eta - \frac 1 2.
\end{equation}
Its eigenfunctions $\psi_k$ are polynomials; we choose their normalization so that $\braket{\phi_i,\psi_j}_{L^2(\R_+)}=\delta_{ij}$.
We defer a more detailed study of these eigenfunctions to Section \ref{sec:mu}.

Now consider the asymptotic expansion of \eqref{eq:ss}.
We substitute the ansatz \eqref{eq:ansatz} in place of $v$, and group terms by order in $\tau$.
The first two terms proceed as in \cite{NRR2}.
At order $e^{\tau/2}$, we find
\begin{equation}
  \m LV_0^+ = 0.
\end{equation}
It follows that $V_0^+=q_0\phi_0$ for some $q_0\in \R$.

To find $q_0$, we introduce the matching between $V^-$ and $V^+$.
We need these two functions to agree to order $t^{-1}$ at $x = t^\eps$.
For the sake of concision, we use the self-similar variables for the matching at $\eta = m(\tau) \coloneqq e^{(\eps-1/2)\tau}$.
From the form of $V^- = V_0^- + t^{-1}V_1^-$,
\begin{equation}
  \label{eq:inner_match}
  V^-(\tau,m(\tau)) = e^{\eps\tau} + \left(-\frac 1 4 e^{3\eps\tau} + \frac 3 4 e^{2\eps\tau} + C_1^-e^{\eps\tau}\right)e^{-\tau} + \m O(e^{-\omega e^{\eps\tau}}).
\end{equation}
With its double-exponential decay, the error term is negligible.
To compare \eqref{eq:inner_match} with $V^+(\tau,m(\tau))$, we Taylor expand $V^+$ in $\eta$, evaluate at $\eta = m(\tau)$, and group the resulting terms by order in $\tau$.
To simplify the resulting expression, we compute its terms sequentially.
Using the explicit form of $V_0^+$, the first terms are
\begin{equation}
  \label{eq:outer_match1}
  V^+(\tau,m(\tau)) = q_0 e^{\eps\tau} + (V_1^+)'(0)e^{(\eps-1/2)\tau} + \m O(e^{(3\eps-1)\tau}).
\end{equation}
Comparing this with \eqref{eq:inner_match}, we see that necessarily $q_0=1$ and $(V_1^+)'(0)=0$.

Having determined $V_0^+$, we turn to $V_1^+$.
The expansion of \eqref{eq:ss} implies:
\begin{equation}
  \label{eq:V1+}
  \left(\m L - \frac 1 2\right)V_1^+ + \frac 3 2 (V_0^+)' + \frac{3\sqrt\pi}{2}V_0^+ = 0.
\end{equation}
This equation has a unique solution, since $\frac 1 2$ is not in the spectrum of $\m L$.
Furthermore, in \cite{NRR2} it is shown that $V_1^+$ satisfies $(V_1^+)'(0)=0$.
Indeed, this condition determines the universal coefficient $3\sqrt\pi$ for $t^{-\frac 1 2}$ in the time-shift $\sigma$.

To compute further terms in $V^+(\tau,m(\tau))$, we require the values
\begin{equation}
  (V_0^+)'''(0) = -\frac 3 2,\quad (V_1^+)''(0) = \frac 3 2.
\end{equation}
The latter follows from \eqref{eq:V1+} and $V_1^+(0)=(V_1^+)'(0)=0$.
Then:
\begin{equation}
  \label{eq:outer_match2}
  V^+(\tau,m(\tau)) = e^{\eps\tau} + \left(-\frac 1 4 e^{3\eps\tau} + \frac 3 4 e^{2\eps\tau}\right)e^{-\tau} + (V_2^+)'(0)\tau e^{(\eps-1)\tau} + \m O(e^{(\eps-1)\tau}).
\end{equation}
Again comparing with \eqref{eq:inner_match}, we find $(V_2^+)'(0)=0$.

At order $\tau e^{-\tau/2}$ in \eqref{eq:ss}, we have
\begin{equation}
  (\m L-1)V_2^+ - \mu V_0^+ = 0.
\end{equation}
Expanding $V_2^+$ in the eigenbasis of $\m L,$ we explicitly find $V_2^+ = -\mu \phi_0 + q_2 \phi_1$ for some $q_2\in \R$.
Using the condition derived above,
\begin{equation}
  0 = (V_2^+)'(0) = -\mu - \frac 3 2 q_2.
\end{equation}
So $q_3 = -\frac 2 3 \mu$ and
\begin{equation}
  V_2^+ = -\mu \left(\phi_0 + \frac 2 3 \phi_1\right).
\end{equation}

Finally, at order $e^{-\tau/2}$ we have
\begin{equation}
  (\m L-1)V_3^+ + V_2^+ + \frac 3 2(V_1)' + \frac{3\sqrt\pi}{2}V_1^+ - \frac{3\sqrt\pi}{2}(V_0^+)' + (\mu - \shiftone) V_0^+ = 0.
\end{equation}
Using the explicit forms for $V_0^+$ and $V_2^+$, we write this as
\begin{equation}
  \label{eq:V3+}
  (\m L-1)V_3^+ = \frac 2 3 \mu \phi_1 - \frac 3 2 (V_1^+)' - \frac{3\sqrt\pi}{2}V_1^+ + \frac{3\sqrt\pi}{2}\phi_0' + \shiftone\phi_0.
\end{equation}

Now, by the definition of the adjoint eigenfunctions, $\psi_1$ is $L^2(\R_+)$-orthogonal to the range of $\m L-1$.
In fact, \eqref{eq:V3+} has a solution \emph{if and only if} $\psi_1$ is orthogonal to the right hand side.
That is, if and only if
\begin{equation}
  \label{eq:mu_implicit}
  \Braket{\frac 2 3 \mu \phi_1 - \frac 3 2 (V_1^+)' - \frac{3\sqrt\pi}{2}V_1^+ + \frac{3\sqrt\pi}{2}\phi_0',\psi_1}_{L^2(\R_+)} = 0.
\end{equation}
Here we have used $\braket{\phi_0,\psi_1}=0$, so the $\shiftone$-term drops out.
This equation determines the unique value $\mu_*$ that permits us to match $V^-$ and $V^+$ with sufficiently high accuracy.
We explicitly compute $\mu_*$ in Section \ref{sec:mu}, where we show:
\begin{lemma}
  \label{lem:mu}
  Equation \eqref{eq:mu_implicit} implies $\mu_* = \frac 9 8 \left(5 - 6\log 2\right).$
\end{lemma}
\noindent This is the value found by Berestycki, Brunet, and Derrida in \cite{BBD} for a closely related problem.

Having determined $\mu_*$, at least implicitly, we return to the equation for $V_3^+$.
Although we have guaranteed the existence of a solution to \eqref{eq:V3+}, we do not have uniqueness.
Indeed, $\m L-1$ has nullspace spanned by $\phi_1$, so we have only determined $V_3^+$ up to a multiple of $\phi_1.$
More precisely, let $\bar{V}_3^+$ denote a particular solution to \eqref{eq:V3+} when $\shiftone=0$.
Then a general solution to \eqref{eq:V3+} has the form
\begin{equation}
  \label{eq:V3+}
  V_3^+ = \bar{V}_3^+ - \shiftone\phi_0 + q_3\phi_1.
\end{equation}
for some $q_3\in \R.$
For the moment, we leave $q_3$ undetermined.
In the proof of Theorem \ref{thm:main}, we will use this free constant to push the accuracy of \eqref{eq:main} below $\m O(t^{-1})$.
We will see that $q_3$ depends on the initial data $u_0$.

For the moment, fix $q_3\in \R$, and consider $V^+(\tau,m(\tau))$.
We have now defined all terms in $V^+$, so
\begin{equation}
  \label{eq:outer_match3}
  V^+(\tau,m(\tau)) = e^{\eps\tau} + \left(-\frac 1 4 e^{3\eps\tau} + \frac 3 4 e^{2\eps\tau} + (V_3^+)'(0) e^{\eps\tau}\right)e^{-\tau} + \m O(e^{(4\eps-3/2)\tau}).
\end{equation}
Comparing this expansion with \eqref{eq:inner_match}, we require $(V_3^+)'(0)=C_1^-$.
We therefore choose $\shiftone$ so that
\begin{equation}
  \label{eq:balance}
  C_1^- = (V_3^+)'(0) = (\bar V_3^+)'(0) - \shiftone - \frac 3 2 q_3.
\end{equation}
Thus $\shiftone$ depends on $u_0$ through $q_3$.
Note also that the absence of a pure $e^{-\tau}$ term in \eqref{eq:outer_match3} forces $C_0^-  =0$ in \eqref{eq:V1-behavior+}.
This condition motivates the form of $V_1^-$ given by Lemma \ref{lem:ODE}.

\section{Computation of $\mu_*$}
\label{sec:mu}

We now offer an explicit computation of the coefficient $\mu_*$ determined by \eqref{eq:mu_implicit}.
We ultimately recover the value found by Berestycki, Brunet, and Derrida in \cite{BBD}.

Recalling that $\braket{\phi_1,\psi_1}=1$, we rewrite \eqref{eq:mu_implicit} as
\begin{equation}
  \label{eq:mu}
  \mu_* = \frac 3 2 \Braket{\frac 3 2 (V_1^+)' + \frac{3\sqrt\pi}{2}V_1^+ - \frac{3\sqrt\pi}{2}\phi_0',\psi_1}.
\end{equation}
From the explicit form of $\phi_0$, we can compute $\braket{\phi_0',\psi_1} = -\frac 1 {\sqrt\pi}$.
Also, from \eqref{eq:V1+} we have
\begin{equation}
  (L-1)V_1^+ + \frac 3 2 \phi_0' + \frac{3\sqrt\pi}{2}\phi_0 = -\frac 1 2 V_1^+.
\end{equation}
Since $\psi_1$ is orthogonal to the range of $L-1$,
\begin{equation}
  \braket{V_1^+,\psi_1} = -\Braket{3\phi_0'+3\sqrt\pi \phi_0,\psi_1} = -3\braket{\phi_0',\psi_1} = \frac{3}{\sqrt\pi}.
\end{equation}
Now let $\theta$ denote the unique Dirichlet solution to $\left(L-\frac 1 2\right)\theta = \phi_0'$.
Then \eqref{eq:V1+} implies $V_1^+ = -\frac 3 2 \theta + 3\sqrt\pi \phi_0$.
Hence
\begin{equation}
  \Braket{(V_1^+)',\psi_1} = -\frac 3 2 \braket{\theta',\psi_1} + 3\sqrt\pi \braket{\phi_0',\psi_1} = - \frac 3 2 \braket{\theta',\psi_1} - 3.
\end{equation}
Combining these calculations, \eqref{eq:mu} yields
\begin{equation}
  \label{eq:mu_inter}
  \mu_* = \frac 9 4 - \frac{27}{8}\braket{\theta',\psi_1}.
\end{equation}

Before examining $\theta$, we first relate $\phi_k$ and $\psi_k$ to the well-known Hermite polynomials.
For $n\in \Z_{\geq 0}$, let
\begin{equation}
  H_n(\eta)\coloneqq (\eta - 2\partial_\eta)^n1.
\end{equation}
Then $H_n$ is a scaled variant of the $n^{\text{th}}$ Hermite polynomial.
From the definition of $\phi_k$ and well-known properties of the Hermite polynomials, it is straightforward to check that
\begin{equation}
  \phi_k(\eta) = 4^{-k}H_{2k+1}(\eta)e^{-\eta^2/4},\quad \psi_k(\eta) = \frac{1}{2\sqrt\pi (2k+1)!}H_{2k+1}(\eta).
\end{equation}

We now express $\theta$ in the $\{\phi_k\}$ basis:
\begin{equation}
  \theta = \sum_{k\geq 0} c_k\phi_k. 
\end{equation}
for $c_k\in \R$.
By the defining equation for $\theta$,
\begin{equation}
  \left(L-\frac 1 2\right)\theta = \sum_k c_k\left(k-\frac 1 2\right)\phi_k = \phi_0'.
\end{equation}
Taking the inner product with the dual basis, orthogonality implies 
\begin{equation}
  c_k = \frac 1 {k-1/2}\braket{\phi_0',\psi_k}.
\end{equation}

Integrating by parts, we have
\begin{equation}
  \braket{\theta',\psi_1} = -\braket{\theta,\psi_1'} = -\sum_k c_k\braket{\phi_k,\psi_1'}.
\end{equation}
But
\begin{align}
  \braket{\phi_k,\psi_1'} &= \int_{\R_+} 4^{-k}H_{2k+1}(\eta)e^{-\eta^2/4}\frac 1 {4\sqrt{\pi}}(\eta^2-2)\;d\eta\\
                          &= -(2k+1)! \, 4^{-k}\int_{\R_+} \left(1-\frac {\eta^2} 2\right)e^{-\eta^2/4}\frac 1 {2\sqrt\pi (2k+1)!}H_{2k+1}(\eta)\;d\eta = -(2k+1)!\,4^{-k}\braket{\phi_0',\psi_k}.
\end{align}
Hence
\begin{equation}
  \label{eq:sum_init}
  \braket{\theta',\psi_1} = -\sum_k c_k\braket{\phi_k,\psi_1'} = \sum_{k\geq 0} \frac{(2k+1)!}{4^k(k-1/2)}\braket{\phi_0',\psi_k}^2.
\end{equation}

Next, we claim that
\begin{equation}
  \label{eq:claim}
  \braket{\phi_0',\psi_k} = \frac{(-1)^k}{\sqrt \pi (2k-1)\,k!}\quad \text{for all }k\geq 0.
\end{equation}

\begin{proof}[Proof of Claim.]
  First note that $\phi_0' = -\frac 1 2 H_2$, so
  \begin{equation}
    \label{eq:braket}
    \braket{\phi_0',\psi_k} = -\frac 1 {4\sqrt\pi (2k+1)!}\int_{\R_+}H_2H_{2k+1}e^{-\eta^2/4}\;d\eta.
  \end{equation}
  Using the definition of $H_n$, and integrating by parts, we find
  \begin{align}
    \int_{\R_+}H_2H_{2k+1}e^{-\eta^2/4}\;d\eta &= -2\int_{\R_+}H_2H_{2k}\partial_\eta(e^{-\eta^2/4})\;d\eta - 2\int_{\R_+}H_2H_{2k}'e^{-\eta^2/4}\;d\eta\\
                                               &= 2\int_{\R_+}H_2'H_{2k}e^{-\eta^2/4} + 2H_2(0)H_{2k}(0).
  \end{align}
  Repeating this procedure, we further find
  \begin{equation}
    2\int_{\R_+}H_2'H_{2k}e^{-\eta^2/4}\;d\eta = 4\int_{\R_+}H_2''H_{2k-1}e^{-\eta^2/4}\;d\eta = 8H_2''(0)H_{2k-2}(0).
  \end{equation}
  Using the explicit form for $H_2$, this work yields
  \begin{equation}
    \int_{\R_+}H_2H_{2k+1}e^{-\eta^2/4}\;d\eta = -4 H_{2k}(0) + 16 H_{2k-2}(0).
  \end{equation}
  From standard formul\ae~for the Hermite polynomials,
  \begin{equation}
    H_{2k}(0) = (-1)^k 2^k (2k-1)!!,\quad H_{2k-2}(0) = (-1)^{k-1}2^{k-1}(2k-3)!!.
  \end{equation}
  So
  \begin{equation}
    \int_{\R_+}H_2H_{2k+1}e^{-\eta^2/4}\;d\eta = -(-1)^k2^k[4(2k-1)+8](2k-3)!! = -4(-1)^k 2^k(2k+1) (2k-3)!!.
  \end{equation}
  By \eqref{eq:braket}, we obtain \eqref{eq:claim}:
  \begin{align}
    \braket{\phi_0',\psi_k} = \frac{(-1)^k2^k(2k-3)!!}{\sqrt\pi (2k)!} = \frac{(-1)^k 2^k}{\sqrt\pi (2k)!!(2k-1)} = \frac{(-1)^k}{\sqrt\pi (2k-1)k!}.
  \end{align}
\end{proof}

Combining \eqref{eq:sum_init} and \eqref{eq:claim}, we obtain the series representation
\begin{equation}
  \label{eq:sum_rep}
  \braket{\theta',\psi_1} = \frac{2}{\pi}\sum_{k\geq 0} \frac{(2k+1)!}{4^k(k!)^2(2k-1)^3}.
\end{equation}

\begin{lemma}
  \label{lem:sum}
  We have
  \begin{equation}
    \label{eq:sum}
    \sum_{k\geq 0} \frac{(2k+1)!}{4^k(k!)^2(2k-1)^3} = \frac\pi 2 (2\log 2 - 1).
  \end{equation}
\end{lemma}

Before proving Lemma \ref{lem:sum}, we use it to conclude the computation of $\mu_*$.

\begin{proof}[Proof of Lemma \textup{\ref{lem:mu}}.]
  From \eqref{eq:sum_rep} and \eqref{eq:sum}, $\braket{\theta',\psi_1} = 2\log 2 - 1$.
  Therefore \eqref{eq:mu_inter} implies
  \begin{equation}
    \mu_* = \frac 9 4 - \frac{27}{8}\braket{\theta',\psi_1} = \frac 9 4 - \frac{27}{8}(2\log 2 - 1) = \frac 9 8 (5 - 6\log 2).\qedhere\tag*{$\square$}
  \end{equation}
\end{proof}

We have thus reduced the problem to computing a sum in closed form.
\begin{proof}[Proof of Lemma \textup{\ref{lem:sum}}.]
  Let $S$ denote the sum in \eqref{eq:sum}.
  We first note that the sum converges by Stirling's formula.
  Using $(2k+1)!=(2k+1)\cdot (2k)!$ and $(2k!)(k!)^{-2}=\binom{2k}{k}$, we have
  \begin{equation}
    S = \sum_{k\geq 0} \binom{2k}{k}\frac{2k+1}{4^k(2k-1)^3} = \sum_k \binom{2k}{k} 2^{-2k}\left[\frac 1 {(2k-1)^2}+\frac 2 {(2k-1)^3}\right].
  \end{equation}
  We view this sum as a power series evaluated at $x = \frac 1 2$.
  As noted in \cite{Lehmer}, the binomial theorem implies
  \begin{equation}
    \sum_{k\geq 0} \binom{2k}{k}x^{2k} = \frac 1 {\sqrt{1-4x^2}}\quad \text{for }x\in \left[-\frac 1 2, \frac 1 2\right).
  \end{equation}
  To obtain negative powers of $2k-1$, we repeatedly divide by powers of $x$ and integrate, so that we always integrate terms of the form $x^{2k-2}$.
  We move the constant term in the sum to the right-hand-side, to ensure integrability.
  So:
  \begin{equation}
    \sum_{k\geq 1}\binom{2k}{k}\frac{x^{2k-1}}{2k-1} = \int_0^x\frac{(1-4y^2)^{-1/2}-1}{y^2}\;dy = \frac{4x^2+\sqrt{1-4x^2}-1}{x\sqrt{1-4x^2}}.
  \end{equation}
  Repeatedly dividing by $x$ and integrating, we find:
  \begin{align}
    \sum_{k\geq 1}\binom{2k}{k}\frac{x^{2k}}{(2k-1)^2} &= \sqrt{1-4x^2}-1+2x\arcsin(2x),\\
    \sum_{k\geq 1}\binom{2k}{k}\frac{x^{2k}}{(2k-1)^3} &= 1-\sqrt{1-4x^2} - 2x\arcsin(2x) + 2x\int_0^x\frac{\arcsin(2y)}{y}\;dy.
  \end{align}
  The integrations induce convergence at the right endpoint $x = \frac 1 2$.
  Evaluating there and restoring the constant terms, we obtain
  \begin{equation}
    S = \left[1 - 1 + \frac \pi 2 + 2\left(-1 + 1 - \frac\pi 2 + \int_0^1\frac{\arcsin y}{y}\;dy\right)\right] = -\frac \pi 2 + 2\int_0^1\frac{\arcsin y}{y}\;dy.
  \end{equation}

  Now, the integrand $\frac{\arcsin y}{y}$ has no elementary antiderivative, so we use contour integration to compute the definite integral.
  We first change variables and integrate by parts:
  \begin{equation}
    \int_0^1\frac{\arcsin y}{y}\;dy = \int_0^{\frac\pi 2}u\cot u\;du = u\log(\sin u)\bigg|_0^{\frac\pi 2} - \int_0^{\frac\pi 2}\log(\sin u)\;du = -\int_0^{\frac \pi 2}\log(\sin u)\;du.
  \end{equation}
  By trigonometric symmetries,
  \begin{equation}
    \int_0^{\pi/2}\log(\sin u)\;du = \int_0^{\pi/2}\log(\cos u)\;du = \frac 1 2 \int_{-\pi/2}^{\pi/2}\log(\cos u)\;du.
  \end{equation}

  Now consider the open half-strip $D\subset \C$ in the upper half-plane bounded by the lines $\Re z = \pm\frac{\pi}{2}$.
  Within $D$, the function $f(z)\coloneqq \log(e^{2iz}+1)$ is analytic (using the standard branch of the logarithm).
  Furthermore, since the complex arguments of $e^{iz}$ and $(e^{iz}-e^{-iz})$ stay within $\left(-\frac \pi 2,\frac\pi 2\right)$ in $D$, we have
  \begin{equation}
    f(z) = \log\left[e^{iz}(e^{iz}+e^{-iz})\right] = \log(e^{iz}) + \log(e^{iz}+e^{-iz}) = iz + \log 2 + \log(\cos z).
  \end{equation}
  By a standard limiting argument,
  \begin{equation}
    \int_{\partial D}f(z)\;dz = 0.
  \end{equation}
  On the other hand, $f\left(-\frac\pi 2 +it\right) = f\left(\frac\pi 2 +it\right)$ for $t>0$, so the contributions from the vertical rays in $\partial D$ cancel in $\int_{\partial D}f$.
  Thus
  \begin{equation}
    0=\int_{-\pi/2}^{\pi/2}f(z)\;dz = \int_{-\pi/2}^{\pi/2}[iz + \log 2 + \log(\cos z)]\;dz = \pi \log 2 + \int_{-\pi/2}^{\pi/2}\log(\cos u)\;du.
  \end{equation}
  Hence $\int_{-\pi/2}^{\pi/2}\log(\cos u)\;du = -\pi \log 2$, and
  \begin{equation}
    \int_0^1\frac{\arcsin y}{y}\;dy = \frac \pi 2 \log 2.
  \end{equation}
  Finally, this implies
  \begin{equation}
    S = -\frac\pi 2 + 2\int_0^1\frac{\arcsin y}{y}\;dy = \frac \pi 2(2\log 2 - 1)
  \end{equation}
  as claimed.
\end{proof}

\section{Complete front asymptotics}
\label{sec:complete}

We now generalize the asymptotic methods in Section \ref{sec:approx} to describe the behavior of $u$ to all orders in $t$.
We make \eqref{eq:shift_complete_intro} and \eqref{eq:u_complete_intro} precise, and present the method for their derivation.
However, we do not rigorously prove the full expansion.
Nonetheless, we expect that the proof in Section \ref{sec:proof} can be generalized to verify our proposed asymptotic form.



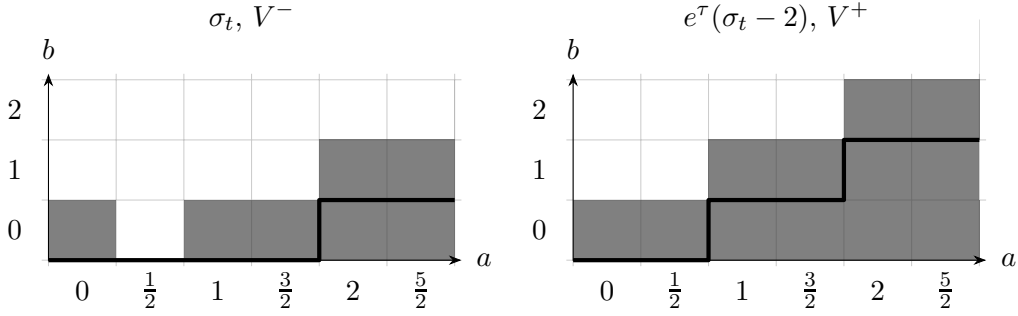
\begin{figure}
  \centering

  \begin{tikzpicture}[xscale = 0.9, yscale = 0.8]
    \def\X{6};
    \def\Y{3};
    \def\T{-7.75};
    \def\S{0.1};
    \def\SS{0.03}
    \coordinate (x axis) at ({\X+\S},0);
    \coordinate (y axis) at (0,{\Y+\S});

    \fill[gray] (0,0) rectangle (\X,1);
    \fill[gray] (2,1) rectangle (\X,2);
    \fill[gray] (4,2) rectangle (\X,3);
    \fill[gray] (6,3) rectangle (\X,4);
    
    \draw[step = 1, darkgray, very thin, opacity = 0.28] ({-\S},{-\S}) grid ({\X+\S},{\Y+\S});
    
    \draw[Stealth - Stealth] (y axis) -- (0,0) -- (x axis);
    \foreach \x/\xtext in {0.5/0 , 1.5/\frac 1 2 , 2.5/1 , 3.5/\frac 3 2, 4.5/2 , 5.5/\frac 5 2 }
    \draw (\x,-0.5) node {$\xtext$};
    \foreach \y/\ytext in {0.5/0, 1.5/1, 2.5/2}
    \draw (-0.5,\y) node {$\ytext$};
    \draw (x axis) node[outer sep = 2pt, right] {$a$};
    \draw (y axis) node[outer sep = 2pt, above] {$b$};
    
    \draw [ultra thick, line join = round] (0,0) -- (2,0) -- (2,1) -- (4,1) -- (4,2) -- (6,2);

    \draw[anchor = mid] (3,4) node {$e^\tau(\sigma_t-2),\,V^+$};

    \fill[gray] (\T,0) rectangle ({\T+1},1);
    \fill[gray] ({\T+2},0) rectangle ({\T+\X},1);
    \fill[gray] ({\T+4},1) rectangle ({\T+\X},2);
    \fill[gray] ({\T+6},2) rectangle ({\T+\X},3);
    
    \draw[step = 1, darkgray, very thin, opacity = 0.28, shift={(\T,0)}] ({-\S},{-\S}) grid ({\X+\S},{\Y+\S});
    
    \draw[Stealth - Stealth] (\T,{\Y+\S}) -- (\T,0) -- ({\T+\X+\S},0);
    \foreach \x/\xtext in {{\T+0.5}/0 , {\T+1.5}/\frac 1 2 , {\T+2.5}/1 , {\T+3.5}/\frac 3 2, {\T+4.5}/2 , {\T+5.5}/\frac 5 2 }
    \draw (\x,-0.5) node {$\xtext$};
    \foreach \y/\ytext in {0.5/0, 1.5/1, 2.5/2}
    \draw ({\T-0.5},\y) node {$\ytext$};
    \draw ({\T+\X+\S},0) node[outer sep = 2pt, right] {$a$};
    \draw (\T,{\Y+\S}) node[outer sep = 2pt, above] {$b$};
    
    \draw [ultra thick, line join = round] (\T,0) -- ({\T+4},0) -- ({\T+4},1) -- ({\T+6},1);

    \draw[anchor = mid] ({\T+3},4) node {$\sigma_t,\,V^-$};

    
  \end{tikzpicture}
  
  \caption{Asymptotic terms in $\sigma_t$ and $V^\pm$, of order $t^{-a}\log^b t$.
    Nonzero terms are shaded.
    Terms above the bold path are universal, \emph{i.e.} independent of the initial data $u_0$.
  }
   \label{fig:terms}
\end{figure}


As in Section \ref{sec:approx}, we describe the inner expansion $V^-$, outer expansion $V^+$, and front-shift $\sigma$ in successively smaller orders of $t$.
All orders will have the form $t^{-a}\log^b t$ with $a$ a half-integer and $b$ an integer.
To facilitate our discussion, we introduce notation adapted to this structure.
We let the subscript $(a,b)$ denote the coefficient of order $t^{-a}\log^b t$ in an asymptotic expansion in $t$.
Of course, not all terms of the form $t^{-a}\log^b t$ appear: only finitely many factors of $\log t$ accompany any fixed $t^{-a}$.
As we shall see, $V^-$, $V^+$, and $\sigma$ have closely related but distinct expansions in $t$.
To be precise, define:
\begin{equation}
  \label{eq:Omega}
  \begin{split}
  \Omega^- &\coloneqq \left\{(0,0)\right\}\cup \left\{(a,b)\in \frac 1 2 \Z\times \Z;\; a\geq 1,\;0\leq b\leq a-1\right\},\\
  \Omega^+ &\coloneqq \left\{(a,b)\in \frac 1 2 \Z\times \Z;\; a\geq 0,\;0\leq b\leq a\right\},\\
  \Omega^\sigma &\coloneqq \left\{(-1,0),(1,1)\right\}\cup \Omega^+.
\end{split}
\end{equation}
We will argue that
\begin{equation}
  \label{eq:full_asymp}
  \begin{gathered}
    V^-(t,x) = \sum_{(a,b)\in \Omega^-}t^{-a}\log^b t\;V_{a,b}^-(x),\\
    V^+(\tau,\eta) = \sum_{(a,b)\in \Omega^+}\tau^b e^{(1/2-a)\tau}V_{a,b}^+(\eta),\\
  \sigma(t) = \sum_{(a,b)\in \Omega^\sigma}\sigma_{a,b}\,t^{-a}\log^b t.
\end{gathered}
\end{equation}
We have graphically organized this structure in Figure \ref{fig:terms}.
We have emphasized $\sigma_t$ rather than $\sigma$, since the shift always enters into equations through its time derivative.

In \eqref{eq:full_asymp}, equality denotes an asymptotic expansion in powers of $t$.
That is, for any $A\geq 0$ we may truncate the series by omitting terms with $a>A$.
Then each series will equal its left-hand-side up to an error $\smallO\left(t^{-A}\right)$ in the variables $(t,x)$.
We let $\sigma_{(A)}$ denote such a truncation, and likewise $V_{(A)}^\pm$.
That is,
\begin{equation}
  \label{eq:trunc}
  \sigma_{(A)}\coloneqq \sum_{\substack{(a,b)\in \Omega^\sigma\\a\leq A}}\sigma_{a,b}\,t^{-a}\log^b t.
\end{equation}
We propose the following generalization of Theorem \ref{thm:main}:
\begin{proposition}
  \label{prop:complete}
  There exists an asymptotic series of the form \eqref{eq:full_asymp} depending on $u_0$ such that the following holds.
  For any $A\geq 0$, let $\sigma_{(A)}$ and $V_{(A)}^\pm$ be as in \eqref{eq:trunc}.
  Then for any $\gamma>0$, there exist $\eps>0$ and $C_\gamma>0$ depending also on $u_0$ such that
  \begin{equation}
    \abs{u(t,x+\sigma_{(A)}(t)) - e^{-x} \left[V_{(A)}^-(t,x)\tbf{\textup{1}}_{x<t^\eps} + V_{(A)}^+(t,x)\tbf{\textup{1}}_{x\geq t^\eps}\right]} \leq \frac{C_\gamma (1+\abs x)e^{-x}}{t^{A+1/2-\gamma}}\quad\text{on } [3,\infty)\times \R.
  \end{equation}
  Furthermore, for each fixed power of $t$, the terms in $\sigma$ and $V^\pm$ of highest order in $\log t$ are independent of $u_0$.
\end{proposition}

\begin{remark}
  This proposition justifies \eqref{eq:shift_complete_intro} and \eqref{eq:u_complete_intro} in the introduction.
\end{remark}

Note that we have already found the terms in \eqref{eq:full_asymp} with $a\leq 1$.
Using our earlier notation:
\begin{equation}
  \label{eq:previous_terms}
  \begin{split}
  V_{0,0}^- &= V_0^-,\quad V_{1,0}^- = V_1^-,\quad V_{0,0}^+ = V_0^+,\quad V_{\frac 1 2,0}^+ = V_1^+,\quad V_{1,1}^+ = V_2^+,\quad V_{1,0}^+ = V_3^+,\\
  \sigma_{-1,0} &= 2,\quad \sigma_{0,1} = -\frac 3 2 , \quad \sigma_{0,0} = \shiftzero,\quad \sigma_{\frac 1 2,0} = -3\sqrt\pi,\quad \sigma_{1,1} = \mu_*,\quad \sigma_{1,0} = \shiftone.
\end{split}
\end{equation}

In the remainder of this section, we outline the derivation of the expansion \eqref{eq:full_asymp}.
We proceed inductively on orders in $t$.
Suppose we have determined $V^\pm$ and $\sigma$ to order $t^{1/2 - A}$ for some half-integer $A\geq 1$, and they have the form in \eqref{eq:full_asymp}.
We wish to show that \eqref{eq:full_asymp} continues to hold to order $t^{-A}$.

\subsection{The inner expansion}

First consider the inner expansion $V^-$.
Recall that $V^-$ is an approximate solution to
\begin{equation}
  \label{eq:V-}
  v_t - v_{xx} - (\sigma_t-2)(v-v_x) + e^{-x}v^2 = 0.
\end{equation}
We choose $V_{(A)}^-$ to cancel all terms of order $\m O(t^{-A})$ or larger in \eqref{eq:V-}.
Note that the time derivative on $\sigma$ lowers the order of its terms by a factor of $t^{-1}$.
Since $V^-$ has leading order $\m O(1)$, terms of the form $\sigma_{a,b}$ with $a\geq A-1/2$ make $\smallO(t^{-A})$ contributions to \eqref{eq:V-}.
They thus have no influence on the equations for $V_{A,b}^-$.
Rather, these equation depend only on $V_{(A-1)}^-$ and $\sigma_{(A-1)}$.

To find the largest power of $\log t$ paired with $t^{-A}$ in $V_{(A)}^-$, we substitute $V_{(A-1/2)}^-$ into \eqref{eq:V-}.
Since $V_{(A-1/2)}^-$ was chosen to eliminate all term of order $t^{-(A-1/2)}$ or larger, we are left with terms of the form $t^{-a}\log^bt$ with $a\geq A$.
By the inductive hypothesis, $V^-$ and $\sigma$ obey \eqref{eq:full_asymp} up to order $t^{-(A-1/2)}$.
Using Figure \ref{fig:terms}, we can visually track the contributions from $(\sigma_t-2)(v-v_x)$ and $e^{-x}v^2$ by combining appropriate columns of $\sigma_t$ and $V^-$.

For instance, suppose we wish to compute the $\log t$ factors paired with $t^{-2}$.
To do so, we substitute $\sigma_{(3/2)}$ and $V_{(3/2)}^-$ for $\sigma$ and $v$ in \eqref{eq:V-}.
Hence Figure \ref{fig:terms} holds for $\sigma_t$ up to $a = \frac 5 2$, and for $V^-$ up to $a = \frac 3 2.$
We examine \eqref{eq:V-} term-by-term, to find the factors of $\log t$ at order $a=2$.

The time derivative $\partial_tV^-$ will generate no logarithmic factors, since $V^-$ has none at order $t^{-1}$.
The spatial derivative $\partial_{xx}V^-$ can be ignored, as it does not generate any term of order $t^{-2}$ when we plug in $V_{(3/2)}^-$.
To handle the product $(\sigma_t-2)[V^--(V^-)']$, we combine known columns whose $a$-values sum to $2$.
Of these, only the pairing
\begin{equation}
  (-\sigma_{1,1}t^{-2}\log t)[V_{0,0}^- - (V_{0,0}^-)']
\end{equation}
generates a factor of $\log t$.
Applying an identical approach to $e^{-x}(V^-)^2$, we see that it contributes no logarithmic factors, since $V_{(3/2)}^-$ has no such factors.
Therefore
\begin{equation}
  V_{(2)}^-(t,x) = V_{(1)}^-(t,x) + t^{-2}\log t\, V_{2,1}^-(x) + t^{-2}V_{2,0}^-(x).
\end{equation}

In general, the above argument show that the leading $\log t$ term at order $t^{-A}$ is due to
\begin{align}
  &-\der{}{t}(\sigma_{A-1,\floor{A}-1}t^{-(A-1)}\log^{\floor{A}-1} t)[V_{0,0}^- - (V_{0,0}^-)']\\
  &\hspace{7cm}\sim (A-1)\sigma_{A-1,\floor{A}-1} t^{-A}\log^{\floor{A}-1} t\;[V_{0,0}^- - (V_{0,0}^-)'].
\end{align}
We must therefore include a term of the form $V_{A,\floor A - 1}^-$ in $V_{(A)}^-$.
Naturally, all lower powers of $\log t$ appear as well, so
\begin{equation}
  V_{(A)}^-(t,x) = V_{(A-1/2)}^-(t,x) + \sum_{b=0}^{\floor A - 1}t^{-A}\log^b t\;V_{A,b}^-(x),
\end{equation}
as predicted by \eqref{eq:full_asymp}.
If we substitute $V_{(A)}^-$ in \eqref{eq:V-}, only $\smallO(t^{-A})$ terms remain.
That is, $NL[V_{(A)}^-] = \smallO(t^{-A})$, where $NL$ is the nonlinear operator in \eqref{eq:V-}.

Recall, however, that further constraints on $V^-$ are necessary.
In particular, $V^-$ must decay as $x\to-\infty$, and must match well with $V^+$.
Now, \eqref{eq:V-} implies that $V_{A,b}^-$ solves an inhomogeneous linear ODE of the form
\begin{equation}
  \label{eq:V-gen}
  -V'' + 2e^{-x}V_{0,0}^-V = F,
\end{equation}
where $F$ is some combination of the functions $e^{-x},V_{a,b}^-,$ and $V_{a,b}^--(V_{a,b}^-)'$ with $a<A$.
We can easily verify that $F=\m O(e^x)$ as $x\to-\infty$ and that $F$ grows polynomially as $x\to +\infty$.
We now desire a solution to \eqref{eq:V-gen} decaying at $-\infty$ and lacking a constant term in its polynomial expansion at $+\infty$ (in order to match $V^+$).
The proof of Lemma \ref{lem:ODE} can be adapted to show the existence of a unique solution to \eqref{eq:V-gen} satisfying these boundary conditions.
We have thus uniquely determined the inner expansion $V_{(A)}^-$, and it conforms to \eqref{eq:full_asymp}.

\subsection{The outer expansion and shift}

We now determine the next terms of $V_{(A)}^+$ and $\sigma_{(A)}$.
Recall that by the Dirichlet condition, order $t^{-A}$ terms in $V^+$ variables correspond to order $e^{(1/2-A)\tau}$ terms in the self-similar variables $(\tau,\eta)$.
We therefore assume $V^+$ obeys \eqref{eq:full_asymp} up to order $e^{(1-A)\tau}$, and wish to continue the pattern to order $e^{(1/2-A)\tau}$.
Likewise, we assume $\sigma$ obeys \eqref{eq:full_asymp} to order $t^{1/2-A}$, and seek to continue its pattern to order $t^{-A}$.

In self-similar variables, $V^+$ is an approximate Dirichlet solution to
\begin{equation}
  \label{eq:V+}
   v_\tau - v_{\eta\eta} - \frac{\eta}{2}v_\eta + e^{\tau}(\sigma_t-2)\left(e^{-\tau/2}v_\eta - v\right) + \exp\left(\tau - \eta e^{\tau/2}\right)v^2 = 0.
\end{equation}
Furthermore, $V^+$ must agree with $V^-$ at the matching point $x=t^\eps$ to high order.
Because we have determined $V_{(A)}^-$, the matching criteria for $V_{(A)}^+$ are fixed.
Let $\m{NL}$ denote the nonlinear operator in \eqref{eq:V+}.
We choose $\sigma_{(A)}$ and $V_{(A)}^+$ to ensure the existence of an approximate solution to \eqref{eq:V+} such that $\m{NL}[V^+]=\smallO(e^{(1/2-A)\tau})$ and $\abs{V^+-V^-} = \smallO(t^{-A})$ at $x=t^\eps$.

Given $V_{(A-1/2)}^+$, let us consider the equations for $V_{A,b}^+$.
As in Section \ref{sec:approx}, we may neglect the nonlinear term in \eqref{eq:V+}, as it decays super-exponentially as $\tau\to\infty$.
The $e^\tau$ prefactor before $\sigma_t$ means $\sigma_{A,b}$ affects $V_{A,b}^+$.
In particular, $V_{A,b}^+$ must solve an inhomogeneous linear equation of the form
\begin{equation}
  \label{eq:V+term}
  (\m L-A)V - A \sigma_{A,b}V_{0,0}^+ = G,
\end{equation}
where $G$ depends on the ``larger'' terms: $V_{a,b'}^+$ and $\sigma_{a,b'}$ with $a<A$ or $a=A$ and $b'>b$.
We therefore iteratively determine $V_{A,b}^+$ and $\sigma_{A,b}$, beginning with the largest value of $b$.
We divide our analysis into two cases, determined by the Dirichlet invertibility of $\m L - A$.

First suppose $A\not\in \Z$, so $\m L - A$ is Dirichlet invertible.
We substitute $V_{(A-1/2)}^+$ and $\sigma_{(A-1/2)}$ into \eqref{eq:V+}, and use Figure \eqref{fig:terms} as before to find ``leftover'' terms.
In general, we observe terms of size $\tau^{A-1/2} e^{-A\tau}$ and smaller.
We therefore require terms of the form $V_{A,b}^+$ with $0\leq b\leq A-1/2=\floor{A}$, which agrees with \eqref{eq:full_asymp}.
Recalling that $V_{0,0}^+=\phi_0,$ the equation for $V_{A,\floor A}^+$ has the form
\begin{equation}
  \label{eq:V+term_simple}
  (\m L - A)V - A\sigma_{A,\floor{A}} \phi_0 = G,
\end{equation}
where $G$ depends only on already-determined terms.
This equation has a unique solution for any $\sigma_{A,\floor{A}}$, and changing $\sigma_{A,\floor{A}}$ changes $V_{A,\floor{A}}^+$ by a multiple of $\phi_0$.
As in Section \ref{sec:approx}, an accurate matching of $V^+$ with $V^-$ requires a prescribed value of $\partial_\eta V_{A,\floor{A}}^+(0)$.
In fact, since $V^-$ has no term of order $t^{-A}\log^{\floor A}t$, we need $\partial_\eta V_{A,\floor{A}}^+ (0)=0$.
We therefore choose $\sigma_{A,\floor{A}}$ so that $\partial_\eta V_{A,\floor{A}}^+(0)=0$.

Now suppose we have uniquely determined $\sigma_{A,b'}$ and $V_{A,b'}^+$ for $b'>b$.
Then $V_{A,b}^+$ satisfies an equation of the form $(\m L - A)V - A\sigma_{A,b}\phi_0 = G$ for some already-determined $G$.
As above, we uniquely choose $\sigma_{A,b}$ so that $\partial_\eta  V_{A,b}^+(0)$ has the value required by matching (which is not generally 0).
Iterating in $b$, we thus uniquely determine $\sigma_{(A)}$ and $V_{(A)}^+$ when $A\not\in \Z$.
At each stage, we use the degree of freedom afforded by $\sigma$ to impose a second boundary condition on $V^+$, which permits an accurate matching with $V^-$.

Next suppose $A\in \Z$.
Now $\m L-A$ is not invertible, but rather has one-dimensional kernel and cokernel.
Thus at each stage we have an additional constraint: the inhomogeneity $G$ in \eqref{eq:V+term_simple} must be orthogonal to $\psi_A$, the $A$-eigenfunction of the adjoint operator $\m L^*$.
However, if this constraint is satisfied we obtain a new degree of freedom: $\m L-A$ has nontrivial kernel, so \eqref{eq:V+term_simple} only determines $V$ up to a multiple of the eigenfunction $\phi_A$.
We therefore typically have two constraints and two degrees of freedom, which result in a unique solution.

To be more precise, consider the leading order term $V_{A,A}^+$.
For this term,
\begin{equation}
  G = -\sum_{a=1}^{A-1}(a-1)\sigma_{a,a}V_{A-a,A-a}^+.
\end{equation}
Also, we can inductively show that $V_{a,a}^+$ lies in the span of $\{\phi_0,\ldots,\phi_a\}$ for all $0\leq a<A$.
But $\braket{\phi_a,\psi_A}=0$ when $a<A$, so automatically $\braket{G,\psi_a}=0$.
In fact, in this case we have an explicit solution
\begin{equation}
  V_{A,A}^+ = -\sigma_{A,A}\phi_0 + q_{A,A}\phi_A + \sum_{a=1}^{A-1}G_a\phi_a,
\end{equation}
where $G_a\in \R$ are already determined and $q_{A,A}\in \R$ is a free parameter corresponding to the nontrivial nullspace of $\m L-A$.
Matching with $V^-$ forces $\partial_\eta V_{A,A}^+ (0)=0$, as $V^-$ has no term of order $t^{-A}\log^At$.
We therefore have one constraint on the free parameters $\sigma_{A,A}$ and $q_{A,A}$.

Now consider the equation for $V_{A,A-1}^+$.
It will have the form
\begin{equation}
  \label{eq:V_A,A-1}
  (\m L - A)V - A\sigma_{A,A-1}\phi_0 = \ti G - \sigma_{A,A}\phi_0 - AV_{A,A}^+\eqqcolon G,
\end{equation}
for already-determined $\ti G$.
Indeed, the term $\sigma_{A,A}$ is due to $\partial_t$ acting on the logarithmic prefactor in $\sigma_{A,A}t^{-A}\log^At$.
Likewise, $AV_{A,A}^+$ arises when $\partial_\tau$ acts the polynomial prefactor in $\tau^Ae^{(1/2-A)\tau}V_{A,A}^+$.
For \eqref{eq:V_A,A-1} to have a solution, we must have $\braket{G,\psi_A}=0$.
Recalling that $\braket{\phi_a,\psi_A}=\delta_{aA}$, we have
\begin{equation}
  \braket{G,\psi_A} = \braket{\ti G,\psi_A} - Aq_{A,A}.
\end{equation}
We may therefore choose $q_{A,A}$ to ensure $\braket{G,\psi_A}=0$.
In turn, $q_{A,A}$ determines $\sigma_{A,A}$, through the requirement that $\partial_\eta V_{A,A}^+(0)=0$.

We may repeat this procedure for successively smaller values of $b$.
At each stage, the free coefficient $q_{A,b}$ of $\phi_A$ in $V_{A,b}^+$ is chosen to ensure existence for $V_{A,b-1}^+$.
The matching condition then determines $\sigma_{A,b}$.
This procedure continues until $b=0$.
Then there are no lower-order equations, so $q_{A,0}$ seems undetermined.
This parameter mimics $q_3$ in Section \ref{sec:approx}.
It is undetermined by matching, and effectively controls the component of $\phi_A$ in the difference between the approximate solution $V_{\op{app}}$ and the true solution $v$.
As in the proof of Theorem \ref{thm:v} presented below, we can uniquely choose $q_{A,0}$ to kill this component.
Through this choice, $q_{A,0}$ depends on the initial data $u_0$.
Having fixed $q_{A,0},$ the shift $\sigma_{A,0}$ is determined, and likewise depends on $u_0$.
With these choices, we have completely determined $V_{(A)}^+$ and $\sigma_{(A)}^+$, which have the claimed forms.

\subsection{Universality}

We now consider the universality of terms in \eqref{eq:full_asymp}.

We claim that shift coefficients of the form $\sigma_{A,\floor A}$ with $A\geq \frac 1 2$ are independent of the initial data $u_0$, as are $\sigma_{-1,0}=2$ and $\sigma_{0,1}=-\frac 3 2$.
As a direct result, the outer expansion terms of the form $V_{A,\floor A}^+$ are universal for all $A\geq 0$.
We again argue inductively, so suppose this universality holds up to order $t^{1/2-A}$ for some $A\geq 1$.

All terms in the equation for $V_{A,\floor A}^+$ are linear combinations of universal shift terms and universal $V^+$ terms (or their derivatives).
Furthermore, since $V^-$ has no matching term, the boundary data for $V_{A,\floor A}^+$ is $V_{A,\floor A}^+(0)=\partial_\eta V_{A,\floor A}^+(0)=0$.
When $A\not\in \Z$, the shift $\sigma_{A,\floor A}$ is determined solely by the equation for $V_{A,\floor A}^+$ and its boundary data.
Hence in this case $\sigma_{A,\floor A}^+$ and $V_{A,\floor A}^+$ are universal.

When $A\in \Z$, $\sigma_{A,A}$ also depends on the equation for $V_{A,A - 1}^+$, through the parameter $q_{A,A}$.
The only non-universal term in the equation for $V_{A,A-1}^+$ is $-A\sigma_{A,A-1}\phi_0$.
However, $\braket{\phi_0,\psi_A}=0$, so this term has no effect on the solvability of the equation.
Since $q_{A,A}$ is chosen to ensure this solvability, it is independent of $u_0.$
Thus so are $\sigma_{A,A}$ and $V_{A,A}^+$.
It follows that the claimed terms in $\sigma$ and $V^+$ are independent of $u_0$.

We next argue that inner expansion terms of the form $V_{A,\floor A-1}^-$ are universal.
Indeed, the equation for such a term is
\begin{equation}
  -V'' + 2e^{-x}V_{0,0}^-V = -(A-1)\sigma_{A-1,\floor{A}-1}[V_{0,0}^- - (V_{0,0}^-)'].
\end{equation}
Comparing with \eqref{eq:V1-}, we see that $V_{A,\floor{A}-1}^-$ is a multiple of $V_{1,0}^-$.
The scaling factor is proportional to $\sigma_{A-1,\floor A - 1}$.
Since this shift term is universal, so is $V_{A,\floor A - 1}^-$.

Finally, we note that the shift terms $\sigma_{A,\floor A}$ are universal in a broader sense.
Suppose we change the form of the nonlinearity in \eqref{eq:FKPP}, so the equation becomes
\begin{equation}
  u_t = u_{xx} + f(u)
\end{equation}
with a more general KPP reaction $f$.
Assume that $f'(0)=1$ and $f'(1)<0$.
Then the form of the traveling front $\phi$ will change, but its speed will not, since $f'(0)= 1.$
Our preceding arguments hold for the nonlinearity $f$, and the associated linear operator $\m L$ is \emph{unchanged}.
It follows that the equations for $V_{a,b}^+$ are likewise unchanged.
The inner expansion will change with the front $\phi$, and will affect $V^+$ through the boundary data for $V_{a,b}^+$.
However, terms of the form $V_{A,\floor A}^+$ have no matching $V^-$ term, and are thus independent of the changes to the inner expansion.
It follows that these terms, and their shifts $\sigma_{A,\floor A}$, are independent of the precise form of the nonlinearity.
In effect, they only ``see'' the linear behavior of \eqref{eq:FKPP}.
This strong universality suggests that the special coefficients $\sigma_{A,\floor A}$ may arise in more general pulled front settings.

\section{Proof of main theorem}
\label{sec:proof}

We now proceed with the proof of Theorem \ref{thm:main}.
Following Section \ref{sec:approx}, we work in the shifted frame $x\mapsto x - \sigma(t)$.
We know from \cite{NRR1,Bramson78} that there exists $\shiftzero\in \R$ depending on $u_0$ such that $u(t,x)\to \phi(x)$ as $t\to\infty$.
Without loss of generality, we shift the initial data $u_0$ so that $\shiftzero = 0$.

As in Section \ref{sec:approx}, we primarily study $v = e^xu$.
We will construct $V_{\op{app}}$ and prove:
\begin{theorem}
  \label{thm:v}
  There exists a choice of $q_3$ in \eqref{eq:V3+} depending on the initial data $u_0$ such that the following holds.
  For all $\gamma>0$, there exists $C_\gamma>0$ also depending on $u_0$ such that for all $t\geq 3$ and $x\geq 2 - t^\eps$,
  \begin{equation}
    \label{eq:v_bound}
    \abs{v(t,x) - V_{\op{app}}(t,x)} \leq \frac{C_\gamma(1+\abs{x})}{t^{\frac 3 2 - \gamma}}.
  \end{equation}
\end{theorem}

Our main results follow from Theorem \ref{thm:v}:
\begin{proof}[Proof of Theorem \textup{\ref{thm:main}}.]
  Undo the spatial $k$-shift performed in Section \ref{sec:approx}, and take $u_{\op{app}} = e^{-x}V_{\op{app}}$ and $\psi = e^{-x}V_1^-$.
  To extend our bound from $x\geq 2 - t^\eps$ to all $x\in \R$, note that $u$ and $u_{\op{app}}$ are uniformly bounded, say by $\bar C$.
  Now $t^{\gamma - \frac 3 2}e^{-x}\geq c_\gamma > 0$ when $x\leq 2-t^{\eps}$ and $t\geq 1$, recalling that $\eps$ will depend only on $\gamma$.
  Hence \eqref{eq:main} is trivial when $x\leq 2 - t^\eps$, provided we take $C_\gamma\geq \bar C c_\gamma^{-1}$.
\end{proof}

\begin{proof}[Proof of Corollary \textup{\ref{cor:shift}}.]
  We wish to track the rightmost edge of the level set $\{x;\;u(t,x)=s\}$.
  Theorem \ref{thm:main} shows that $u$ is close to $u_{\op{app}}$, uniformly on rays $x\geq C$.
  Recall that $u_{\op{app}}(t,x+\sigma(t)) = \phi(x)+\m O(t^{-1})$ uniformly on the same rays.
  Hence if we apply Theorem \ref{thm:main} on the ray $x \geq \phi^{-1}(s)-1$, we find $\sigma_s(t) - \sigma(t) = \m O(t^{-1}).$
  Equation \eqref{eq:cor} follows.
\end{proof}

Before proving Theorem \ref{thm:v}, we must first construct the approximate solution $V_{\op{app}}$.
Roughly, we use $V_{\op{app}} = V^-$ when $x<t^\eps$, and $V_{\op{app}} = V^+$ when $x>t^\eps$.
However, we must join $V^\pm$ near $t^\eps$ so that $V_{\op{app}}$ is $C^1$ in space.
Our work in Section \ref{sec:approx} shows:
\begin{equation}
  \label{eq:diff}
  \abs{V^+(t,t^\eps)-V^-(t,t^\eps)} = \m O(t^{4\eps-3/2}).
\end{equation}
To make $V_{\op{app}}$ continuous, we change the spatial argument of $V^-$ by a time-dependent shift $\zeta$ so that
\begin{equation}
  V_0^-(t^\eps + \zeta(t)) +t^{-1}V_1^-(t^\eps + \zeta(t))= V^+(t,t^\eps).
\end{equation}
Since $(V^-)'\to 1$ near $x = t^\eps$ as $t\to\infty$, \eqref{eq:diff} and the construction of $V^\pm$ imply
\begin{equation}
  \zeta(t) = \m O\left(t^{4\eps-3/2}\right),\quad \dot\zeta(t) = \m O\left(t^{4\eps-5/2}\right).
\end{equation}
For the remainder of the paper,
\begin{equation}
  V^-(t,x)\coloneqq V_0^-(x + \zeta(t)) +t^{-1}V_1^-(x + \zeta(t)),
\end{equation}
so that $V^-(t,t^\eps) = V^+(t,t^\eps)$.

We further require $\partial_x(V_{\op{app}})$ to be continuous.
To enforce this, we add a term to $V_{\op{app}}$ whose derivative has a discontinuity precisely canceling that between $\partial_xV^-$ and $\partial_xV^+$.
Let
\begin{equation}
  K(t)\coloneqq \partial_xV^+(t,t^\eps) - \partial_xV^-(t,t^\eps),
\end{equation}
so
\begin{equation}
  K(t) = \m O\left(t^{3\eps-\frac 3 2}\right),\quad \dot{K}(t) = \m O\left(t^{3\eps-\frac 5 2}\right).
\end{equation}
Now define $\varphi\geq 0$ satisfying
\begin{equation}
  -\varphi_{xx}+\varphi = \delta(x-t^\eps),\quad \varphi(0)=\varphi(\infty)=0.
\end{equation}
Explicitly,
\begin{equation}
  \varphi(t,x) =
  \begin{cases}
    e^{-t^\eps}\sinh x & \text{for }0\leq x\leq t^\eps,\\
    \sinh(t^\eps)e^{-x} & \text{for }x>t^\eps.
  \end{cases}
\end{equation}
A term $K\varphi$ would fix the discontinuity.
However, it will be convenient for this perturbation to be compactly supported in space.
Therefore let $\vartheta\in C_c^\infty(\R)$ satisfy $\vartheta(1)=1$ and $\vartheta|_{(0,2)^c}\equiv 0$.
Then let
\begin{equation}
  \label{eq:V_app}
  V_{\op{app}}(t,x) \coloneqq \tbf{1}_{x<t^\eps}V^-(t,x) + \tbf{1}_{x\geq t^\eps}V^+(t,x)+K(t)\vartheta(t^{-\eps} x)\varphi(t,x).
\end{equation}
By the construction of $K$, $\varphi$, and $\vartheta$, $V_{\op{app}}$ is $C^1$ in space.

We are interested in controlling the size of $NL[V_{\op{app}}]$, which measures how badly $V_{\op{app}}$ fails to be a true solution of \eqref{eq:transformed}.
We consider the contributions from $V^-,V^+,$ and $K \vartheta \varphi$ separately.

Recall that $V^-(t,x) = V_0^-(x+\zeta) + t^{-1}V_1^-(x+\zeta)$.
Before the shift by $\zeta$, we constructed $V_0^-$ and $V_1^-$ to eliminate terms up to order $t^{-1}$ in $NL[V_0^- + t^{-1}V_1^-]$.
By the decay of $V_i^-$ on $\R_-$,
\begin{equation}
  NL[V_0^- + t^{-1}V_1^-] = \m O(t^{-3/2}e^x).
\end{equation}
Spatially shifting by $\zeta$ introduces new terms in $NL[V^-]$.
Of these, the most significant is
\begin{equation}
  \left(e^{-x}-e^{-x-\zeta}\right)(V_0^-)^2(x+\zeta),
\end{equation}
which is due to the mismatch in argument between $e^{-x}(V_0^-)^2$ and $(V_0^-)''(x+\zeta)=e^{-x-\zeta}(V_0^-)^2$.
Nonetheless, the decay of $\zeta$ and $V_0^-$ implies
\begin{equation}
  \abs{\left(e^{-x}-e^{-x-\zeta}\right)(V_0^-)^2(x+\zeta)} = \m O(\zeta e^x) = \m O(t^{4\eps-3/2}e^x).
\end{equation}
Therefore
\begin{equation}
  NL[V^-] = \m O(t^{4\eps-3/2} e^x) \quad \text{on }(-\infty,0].
\end{equation}
An identical analysis shows that $NL[V^-] = \m O(t^{4\eps-3/2})$ on $[0,t^\eps]$.

Now consider $V^+$.
If $\m{NL}$ denotes the nonlinear operator in \eqref{eq:ss}, we have constructed $V^+$ so that
\begin{equation}
  \m{NL}(V^+) \leq \m O\left(\tau e^{-\tau} e^{-\eta^2/5}\right)\quad \text{for }\eta\in \R_+.
\end{equation}
However, when we derived \eqref{eq:ss} from \eqref{eq:transformed}, we cleared a common factor of $e^{-\tau}$.
Thus informally: $NL = e^{-\tau}\m{NL}$.
Changing to $(t,x)$, this observation implies
\begin{equation}
  NL[V^+] = \m O\left(\log t\cdot t^{-2}\exp\left[-\frac{x^2}{5t}\right]\right) \quad \text{on }[t^\eps,\infty).
\end{equation}

Finally, the bounds on $V^\pm$ and $K$ show that the correction $K\vartheta\phi$ perturbs $NL[V_{\op{app}}]$ by order $\m O(t^{3\eps-3/2})$ solely on $[0,2t^\eps]$.
Therefore there exists $C_\eps>0$ depending also on $u_0$ (through $\shiftone$) such that
\begin{align}
  \label{eq:NL_bound0}
  \abs{NL[V_{\op{app}}]} \leq C_\eps \left[t^{4\eps-\frac 3 2}e^x\tbf{1}_{(-\infty,0]}(x)+t^{4\eps-\frac 3 2}\tbf{1}_{[0,2t^\eps]}(x)+ t^{\eps-2}\tbf{1}_{[t^\eps,\infty)}(x)\exp\left(-\frac{x^2}{5t}\right)\right].
\end{align}

With the estimate \eqref{eq:NL_bound0} in hand, we are ready to prove Theorem \ref{thm:v}.
We will transform our equation into a Dirichlet problem on the half-line, switch to the self-similar variables
\begin{equation}
  \tau = \log t,\quad \eta = \frac{x}{\sqrt{t}},
\end{equation}
and show that our problem is still dominated by linear theory related to the operator $\m L$ introduced in Section \ref{sec:approx}.

\begin{proof}[Proof of Theorem \textup{\ref{thm:v}}.]
  Let $W \coloneqq v - V_{\op{app}}$.
  Then $W$ satisfies an equation of the form
  \begin{equation}
    \label{eq:W}
    W_t - W_{xx} - \left(\frac 3 {2t} - \frac{3\sqrt\pi}{2t^{3/2}}+\mu_*\frac{\log t}{t^2} + \frac{\shiftone-\mu_*}{t^2}\right)(W-W_x)+e^{-x}(v+V_{\op{app}})W = F,
  \end{equation}
  where by \eqref{eq:NL_bound0},
  \begin{equation}
    \label{eq:NL_bound1}
    \abs{F(t,x)} \leq C_\eps \left[t^{4\eps-\frac 3 2}e^x\tbf{1}_{(-\infty,0]}(x)+t^{4\eps-\frac 3 2}\tbf{1}_{[0,2t^\eps]}(x)+ t^{\eps-2}\tbf{1}_{[t^\eps,\infty)}(x)\exp\left(-\frac{x^2}{5t}\right)\right].
  \end{equation}
  Recall that the constant $C_\eps$ depends on $\eps$ and the initial data $u_0$.
  For the remainder of the proof we suppress such constants with the notation $\lesssim$, which denotes inequality up to a multiplicative constant depending on $\eps$ and $u_0$.
  Similarly, we frequently use larger-than-necessary multiples of $\eps$ in exponents, to simplify presentation.
  Under these conventions, \eqref{eq:NL_bound1} may be written:
  \begin{equation}
    \abs{F(t,x)} \lesssim t^{4\eps-\frac 3 2}e^x\tbf{1}_{(-\infty,0]}(x)+t^{4\eps-\frac 3 2}\tbf{1}_{[0,2t^\eps]}(x)+ t^{4\eps-2}\tbf{1}_{[t^\eps,\infty)}(x)\exp\left(-\frac{x^2}{5t}\right).
  \end{equation}

  We now enforce a Dirichlet condition at $x = -t^\eps$ by subtracting the boundary value from $W$.
  To simplify notation, we then shift $x$ by $t^\eps$, so the Dirichlet condition occurs at $x=0$.
  Therefore define
  \begin{equation}
    \mr W(t,x) \coloneqq W(t,x-t^\eps) - W(t,-t^\eps)\vartheta(x+1),
  \end{equation}
  recalling that $\vartheta\in C_c^\infty(\R)$ satisfies $\vartheta(1) = 1$ and $\vartheta_{(0,2)^c}\equiv 0$.
  To control $W(t,-t^\eps)$, we use the exponential decay of $v$ and $V_{\op{app}}$.
  Indeed, $v,V_0^-,V_1^- \leq Ce^x$ on $\R_-.$
  Thus
  \begin{equation}
    \abs{W(t,-t^\eps)}\lesssim e^{-t^\eps} \lesssim t^{-2}.
  \end{equation}

  It follows that $\mr W$ satisfies
  \begin{equation}
    \label{eq:mrW}
    \mr W_t - \mr W_{xx} + \eps t^{\eps-1} \mr W_x - \left(\frac 3 {2t} - \frac{3\sqrt\pi}{2t^{3/2}}+\mu_*\frac{\log t}{t^2} + \frac{\shiftone-\mu_*}{t^2}\right)(\mr W - \mr W_x)+e^{-x + t^\eps}(\mr v + \mr V_{\op{app}})\mr W = G_1+G_2,
  \end{equation}
  where $\mr v(t,x)\coloneqq v(t,x - t^\eps)$, $\mr V_{\op{app}}$ is analogous, and
  \begin{align}
    \abs{G_1(t,x)} &\lesssim t^{4\eps-\frac 3 2}\tbf{1}_{(0,3t^\eps]}(x) + t^{4\eps-2}\tbf{1}_{[2t^\eps,\infty)}(x)\exp\left[-\frac{(x-t^\eps)^2}{5t}\right],\\
    \abs{G_2(t,x)} &\lesssim t^{-2}\tbf{1}_{[0,2]}(x).
  \end{align}

  Changing to the self-similar variables, we find:
  \begin{equation}
    \mr W_\tau + \left(\m L-\frac 1 2 \right)\mr W + \exp\left(\tau - [\eta-m(\tau)]e^{\tau/2}\right)(\mr v + \mr V_{\op{app}})\mr W = G_1 + G_2 + g(\tau)\mr W_\eta + h(\tau)\mr W 
  \end{equation}
  on the half-line $\eta\in \R_+$ with $\mr W(\tau,0)=0$ for all $\tau\geq 0$.
  We have used the notation
  \begin{align}
    m(\tau)&= e^{(\eps-1/2)\tau},\\
    g(\tau)&\coloneqq \eps e^{(\eps-1/2)\tau} - \frac 3 2e^{-\tau/2} + \frac{3\sqrt\pi}{2}e^{-\tau} - \mu_*\tau e^{-3\tau/2} + (\mu_*-\shiftone)e^{-3\tau/2},\\
    h(\tau)&\coloneqq -\frac{3\sqrt \pi}{2}e^{-\tau/2} + \mu_*\tau e^{-\tau} + (\shiftone-\mu_*)e^{-3\tau/2}.
  \end{align}

  Finally, we symmetrize the operator $\m L$ by multiplying through by $e^{\eta^2/8}$.
  This transforms $\m L$ to
  \begin{equation}
    \m M \coloneqq - \partial_\eta^2 + \left(\frac{\eta^2}{16} - \frac 5 4\right)w.
  \end{equation}
  Then
  \begin{equation}
    w(\tau,\eta)\coloneqq e^{\eta^2/8}\mr W(\tau,\eta)
  \end{equation}
  satisfies
  \begin{equation}
    \label{eq:w}
    w_\tau -\m M w + \exp\left(\tau - [\eta-m(\tau)]e^{\tau/2}\right)(\mr v + \mr V_{\op{app}})w = \sum_{i=1}^3E_i,
  \end{equation}
  where the errors $E_i$ satisfy
  \begin{equation}
    \label{eq:errors}
    \begin{split}
      \abs{E_1}&\lesssim e^{(4\eps-1/2)\tau}\tbf{1}_{[0,3m(\tau)]}(\eta) + e^{(4\eps-1)\tau}e^{-\eta^2/10}\tbf{1}_{[2m(\tau),\infty)}(\eta) \eqqcolon E_{11}+E_{12},\\
      \abs{E_2}&\lesssim e^{-\tau}\tbf{1}_{[0,2e^{-\tau/2}]}(\eta),\\
      E_3 &= g(\tau)\left(w_\eta - \frac\eta 4 w\right) + h(\tau) w.
    \end{split}
  \end{equation}
  Furthermore, the convergence of $v$ to $e^x\phi(x)$ and the definition of $V_{\op{app}}$ imply
  \begin{equation}
    \label{eq:VV_bound}
    \abs{(\mr v + \mr V_{\op{app}})(\tau,\eta)} \lesssim e^{[\eta - m(\tau)]e^{\tau/2}}\tbf{1}_{[0,m(\tau)]}(\eta) + \left(1 + \eta^3 e^{3\tau/2}\right)\tbf{1}_{[m(\tau),\infty)}(\eta).
  \end{equation}

  To control the behavior of $w$, we bootstrap from the bounds obtained in \cite{NRR2}.
  The main result in \cite{NRR2} does not directly apply, as it uses a different shift and approximate solution.
  However, the proof in \cite{NRR2} works in our situation with trivial modifications.
  Thus, as in (4.70) and (4.71) in \cite{NRR2}, we have
  \begin{equation}
    \label{eq:w_bound0}
    \norm{w}_{L^2(\R_+)} + \norm{w}_{L^\infty(\R+)}\lesssim e^{(\eps-1/2)\tau},\quad \abs{w(\tau,\eta)}\lesssim \eta e^{(\eps-1/2)\tau}\text{ for all }(\tau,\eta)\in [0,\infty)\times [0,\infty).
  \end{equation}
  Here we have replaced the exponent $100\gamma$ in \cite{NRR2} with our small parameter $\eps$.

  We use the method of \cite{NRR2} to improve this bound to
  \begin{equation}
    \label{eq:w_bound1}
    \norm{w}_{L^2(\R_+)} + \norm{w}_{L^\infty(\R+)}\lesssim e^{(5\eps-1)\tau},\quad \abs{w(\tau,\eta)}\lesssim \eta e^{(5\eps-1)\tau},
  \end{equation}
  provided $q_3$ in $V_3^+$ is chosen appropriately.
  As we shall see, this control implies Theorem \ref{thm:v}.
  For the initial stage in the proof, take $q_3=0$.

  In the following, let $\{e_k\}_{k\in \Z_{\geq 0}}$ denote orthonormal eigenfunctions of $\m M$.
  Since $\m M$ has the same spectrum as $\m L - \frac 1 2$, we have
  \begin{equation}
    \m M e_k = \left(k - \frac 1 2\right)e_k\quad\text{for }k\in \Z_{\geq 0}.
  \end{equation}
  There exist $c_k\in \R$ (unique up to sign) such that $e_k = c_k \phi_k e^{\eta^2/8}$ and $\norm{e_k}_{L^2(\R_+)}=1$.
  In particular,
  \begin{equation}
    e_0(\eta) = c_0\eta e^{-\eta^2/8},\quad\text{and}\quad e_1 = \frac{c_1}{4} (\eta^3-6\eta)e^{-\eta^2/8}.
  \end{equation}

  We begin by proving:
  \begin{lemma}
    \label{lem:r}
    There exists $r\in \R$ such that
    \begin{equation}
      \label{eq:r}
      \norm{e^{\tau/2}w(\tau,\cdot)-r e_1(\cdot)}_{L^2(\R_+)} \lesssim e^{(2\eps-1/4)\tau}\quad\text{as }\tau\to\infty.
    \end{equation}
  \end{lemma}

  We will use this lemma to choose the final value of $q_3$ in $V_3^+$.

  \begin{proof}
    We first consider the $e_0$-component of $w$.
    By \eqref{eq:w},
    \begin{equation}
      \der{}{\tau}\braket{e_0,w} - \frac 1 2\braket{e_0,w} + \Braket{e_0,\exp\left(\tau - [\eta - m(\tau)] e^{\tau/2}\right)(\mr v + \mr V_{\op{app}})w} = \sum_{i=1}^3 \braket{e_0,E_i}.
    \end{equation}
    The bound \eqref{eq:errors} implies
    \begin{equation}
      \abs{\braket{e_0,E_1}}\lesssim e^{(4\eps-1)\tau},\quad \abs{\braket{e_0,E_2}}\lesssim e^{-2\tau}.
    \end{equation}
    By \eqref{eq:w_bound0}, integration by parts, and Cauchy-Schwarz,
    \begin{equation}
      \abs{\braket{e_0,E_3}}\lesssim e^{(\eps-1/2)\tau}\left(\abs{\braket{e_0,w}} + \abs{\braket{\eta e_0,w}}\right) \lesssim e^{(2\eps-1)\tau}.
    \end{equation}
    Now consider the term $\Braket{e_0,\exp\left(\tau - [\eta - m(\tau)] e^{\tau/2}\right)(\mr v + \mr V_{\op{app}})w}$.
    On the interval $[0,m(\tau)]$, \eqref{eq:VV_bound} and \eqref{eq:w_bound0} imply
    \begin{equation}
      \int_0^{m(\tau)}e_0\exp\left(\tau - [\eta - m(\tau)]e^{\tau/2}\right)\abs{(\mr v + \mr V_{\op{app}})w } \lesssim e^{(\eps-1/2)\tau}\int_0^{m(\tau)} \eta^2\;d\eta \lesssim e^{(4\eps-1)\tau}.
    \end{equation}
    Similarly, 
    \begin{align}
      \int_{m(\tau)}^\infty\exp\left(\tau - [\eta - m(\tau)]e^{\tau/2}\right)e_0\abs{(\mr v + \mr V_{\op{app}})w } &\lesssim  e^{(2\eps+1/2)\tau}\int_0^\infty \eta^2(1+\eta^3 e^{3\tau/2})\exp(-\eta e^{\tau/2})\;d\eta \\
                                                                                                                   &\lesssim e^{(2\eps-1)\tau}\int_0^\infty x^2(1+x^3)e^{-x}\;dx \lesssim e^{(2\eps-1)\tau}.
    \end{align}
    Therefore
    \begin{equation}
      \abs{\Braket{e_0,\exp\left(\tau - [\eta - m(\tau)] e^{\tau/2}\right)(\mr v + \mr V_{\op{app}})w}} \lesssim e^{(4\eps-1)\tau},
    \end{equation}
    and
    \begin{equation}
      \label{eq:e_0_prelim}
      \der{}{\tau}\braket{e_0,w} - \frac 1 2 \braket{e_0,w} = \nu_0(\tau)
    \end{equation}
    with $\abs{\nu_0(\tau)}\lesssim e^{(4\eps-1)\tau}$.

    Now $\lim_{\tau\to\infty}\braket{e_0,w}=0$ by Cauchy-Schwarz (and ultimately by our choice of $\shiftzero$).
    Hence we may integrate \eqref{eq:e_0_prelim} back from $\tau = +\infty$ (with the integrating factor $e^{-\tau/2}$) to obtain
    \begin{equation}
      \label{eq:e_0_bound}
      \abs{\braket{e_0,w}} \leq e^{\tau/2}\int_\tau^\infty e^{-\tau'/2}\abs{\nu_0}(\tau')\;d\tau' \lesssim e^{(4\eps-1)\tau}.
    \end{equation}
    Thus the $e_0$-component of $w$ is as small as desired.

    We next consider the $e_1$-component, which satisfies
    \begin{equation}
      \der{}{\tau}\braket{e_1,w} + \frac 1 2\braket{e_1,w} + \Braket{e_1,\exp\left(\tau - [\eta - m(\tau)] e^{\tau/2}\right)(\mr v + \mr V_{\op{app}})w} = \sum_{i=1}^3 \braket{e_1,E_i}.
    \end{equation}
    An identical argument shows
    \begin{equation}
      \label{eq:e_1_prelim}
      \der{}{\tau}\braket{e_1,w} + \frac 1 2 \braket{e_1,w} = \nu_1(\tau)
    \end{equation}
    with $\abs{\nu_1(\tau)} \lesssim e^{(4\eps-1)\tau}$.
    We rewrite \eqref{eq:e_1_prelim} as
    \begin{equation}
      \der{}{\tau}\left(e^{\tau/2}\braket{e_0,w}\right) = e^{\tau/2}\nu_1(\tau).
    \end{equation}
    Integrating from $\tau=0$, we obtain
    \begin{equation}
      \braket{e_1,w(\tau,\cdot)} = e^{-\tau/2}\left[\braket{e_0,w(0,\cdot)} + \int_0^\infty e^{\tau'/2}\nu_1(\tau')\;d\tau'\right] - e^{-\tau/2}\int_\tau^\infty e^{\tau'/2}\nu_1(\tau')\;d\tau'.
    \end{equation}
    We therefore choose
    \begin{equation}
      r = \braket{e_0,w(0,\cdot)} + \int_0^\infty e^{\tau'/2}\nu_1(\tau')\;d\tau'.
    \end{equation}
    It follows that
    \begin{equation}
      \label{eq:e_1_bound_lem}
      \braket{e_1,w(\tau,\cdot)} = re^{-\tau/2} + \m O(e^{(4\eps-1)\tau}).
    \end{equation}

    We must now control the remaining terms in $w$, namely
    \begin{equation}
      w^\perp \coloneqq w - \braket{e_0,w}e_0 - \braket{e_1,w}e_1.
    \end{equation}
    From \eqref{eq:w},
    \begin{equation}
      \label{eq:w_perp}
      \frac 1 2 \der{}{\tau}\| w^\perp\|^2 + \braket{\m M w^\perp,w^\perp} + e^\tau \int_{\R_+}\exp\left([m(\tau)-\eta] e^{\tau/2}\right)(\mr v + \mr V_{\op{app}})w w^\perp\;d\eta = \sum_{i=1}^3 \braket{E_i,w^\perp}.
    \end{equation}
    Note that \eqref{eq:e_0_bound} and \eqref{eq:e_1_bound_lem} imply the bounds in \eqref{eq:w_bound0} hold for $w^\perp$ as well.
    So
    \begin{align}
      \abs{e^\tau \int_{\R_+}\exp\left([m(\tau)-\eta] e^{\tau/2}\right)(\mr v + \mr V_{\op{app}})w w^\perp\;d\eta} &\lesssim e^{3\eps\tau}\int_{\R_+}\eta^2(1+\eta^3 e^{3\tau/2})\exp\left(- \eta e^{\tau/2}\right)\;d\eta \\
                                                                                                                   &\lesssim e^{(3\eps-3/2)\tau}.
    \end{align}
    Next,
    \begin{equation}
      \abs{\braket{E_{11},w^\perp}} \lesssim e^{(4\eps-1)\tau}\int_0^{2m(\tau)}\eta\;d\eta \lesssim e^{(6\eps-2)\tau}.
    \end{equation}
    Similarly $\abs{\braket{E_2,w^\perp}}\lesssim e^{(\eps-5/2)\tau}$.
    For the $E_{12}$ term, we use a Peter-Paul inequality and keep track of constants:
    \begin{equation}
      \abs{\braket{E_{12},w^\perp}} \leq \eps \|w^\perp\|^2 + C_\eps \norm{E_{12}}^2 \leq \eps \|w^\perp\|^2 + C_\eps e^{(8\eps-2)\tau}.
    \end{equation}

    The $E_3$ term requires a more elaborate analysis.
    First, we easily have
    \begin{equation}
      \abs{\braket{h(\tau)w,w^\perp}} \lesssim e^{-\tau/2}\norm{w^\perp}^2 \lesssim e^{(2\eps-3/2)}.
    \end{equation}
    Now turn to $g(\tau)(w_\eta - \eta w/4)$.
    Integrating by parts,
    \begin{equation}
      \int_{\R_+}w_\eta w^\perp = \int_{\R_+}\left[\frac 1 2\partial_\eta(w^2) +  \braket{e_0,w}\partial_\eta(e_0)w + \braket{e_1,w}\partial_\eta(e_1)w \right].
    \end{equation}
    Now $w$ satisfies Dirichlet boundary conditions and $(e_0)_\eta,(e_1)_\eta\in L^2(\R_+)$, so by Cauchy-Schwarz
    \begin{equation}
      \abs{\int_{\R_+}w_\eta w^\perp} \lesssim \left(\abs{\braket{e_0,w}}+\abs{\braket{e_1,w}}\right)\norm{w} \lesssim e^{(\eps-1)\tau}.
    \end{equation}
    Next, consider
    \begin{equation}
      \int_{\R_+}\eta ww^\perp = \int_{\R_+}\eta (w^\perp)^2 + \braket{e_0,w}\int_{\R_+}\eta e_0 w^\perp + \braket{e_1,w}\int_{\R_+}\eta e_1 w^\perp.
    \end{equation}
    Since $\eta e_0,\eta e_1\in L^2(\R_+)$, the last two terms are $\m O(e^{(\eps-1)\tau})$.
    For the first term, $\int_0^1 (w^\perp)^2\lesssim e^{(2\eps-1)\tau}$, so we have
    \begin{equation}
      \int_{\R_+}\eta (w^\perp)^2 \leq \int_{\R_+}\eta^2 (w^\perp)^2 + \m O(e^{(2\eps-1)\tau}).
    \end{equation}
    Finally, 
    \begin{equation}
      \int_{\R_+}\eta^2(w^\perp)^2 = 16\braket{Mw^\perp,w^\perp} + 20\|w^\perp\|^2 - 16\|(w^\perp)_\eta\|^2 \leq 16\braket{Mw^\perp,w^\perp} + C_\eps  e^{(2\eps-1)\tau}
    \end{equation}
    The prefactor $g(\tau)$ in $E_3$ is eventually positive and of order $e^{(\eps-1/2)\tau}$.
    So for large $\tau$,
    \begin{equation}
      \braket{E_3,w^\perp} \leq Ce^{(\eps-1/2)\tau}\braket{Mw^\perp,w^\perp} + C_\eps e^{(3\eps-3/2)\tau}.
    \end{equation}
    Combining these bounds and using $\braket{\m Mw^\perp,w^\perp}\geq \frac 3 2\norm{w}^2$, we obtain for large $\tau$:
    \begin{equation}
      \frac 1 2 \der{}{\tau}\|w^\perp\|^2 + \left(\frac 3 2 - \eps - Ce^{(\eps-1/2)\tau}\right)\|w^\perp\|^2 \leq C_\eps e^{(4\eps-3/2)\tau},
    \end{equation}
    where $C$ depends on $\gamma$ and $u_0$.
    We can absorb $Ce^{(\eps-1/2)\tau/2}\|w^\perp\|^2$ into the right-hand-side and integrate to obtain
    \begin{equation}
      \label{eq:w_perp_bd}
      \|w^\perp\|^2 \lesssim e^{(4\eps-3/2)\tau}.
    \end{equation}
    Together with \eqref{eq:e_0_bound} and \eqref{eq:e_1_bound_lem}, this bound implies \eqref{eq:r}.
  \end{proof}

  We are now able to set the final value of $q_3$, and thus to fully specify $V_{\op{app}}$.
  We let $q\coloneqq c_1^{-1}r$, and set $q_3 = q$.
  We claim that with this choice,
  \begin{equation}
    \label{eq:q_3_limit}
    \lim_{\tau\to\infty}e^{\tau/2}\abs{\braket{e_1,w}} = 0.
  \end{equation}
  Thus $q_3$ is chosen to kill the $e_1$-component of $w$, much as $\shiftzero$ was chosen in \cite{NRR2} to kill the $e_0$-component.

  To see \eqref{eq:q_3_limit}, we consider how $w$ has changed through the change in $q_3$.
  We use the superscripts $o$ and $n$ to denote the old and new definitions, respectively.
  So $w$ has changed from $w^o$ to $w^n$.
  By the calculations in the proof of Lemma \ref{lem:r}, the changes to $w$ on the interval $[0,2m(\tau)]$ are negligible in $L^2(\R_+)$.
  We therefore focus on the change to $w$ on $[2m(\tau),\infty)$.

  When we increase $q_3$ from 0 to $q$, we must decrease $\shiftone$ by $\frac 3 2 q$ to satisfy \eqref{eq:balance}.
  So
  \begin{equation}
    \sigma^n(t) = \sigma^o(t) - \frac{3q}{2t}.
  \end{equation}
  Evaluating $u$ in the unshifted frame, we have:
  \begin{equation}
    v(t,x)\coloneqq e^x u(t,x+\sigma(t)).
  \end{equation}
  Thus
  \begin{equation}
    v^n(t,x) = e^xu(t,x+\sigma^n(t)) = e^xu\left(t,x+\sigma^o(t)-\frac 3 2 qt^{-1}\right) = e^{\frac 3 2 qt^{-1}}v^o\left(t,x-\frac 3 2 qt^{-1}\right).
  \end{equation}
  We shift $x$ by $t^\eps$, and change to the self-similar variables.
  By the decay of $w$ and the form of $V^+$, we know that $v(\tau,\cdot) = e^{\tau/2}\phi_0(\cdot) + \m O(1)$ in $L^2$.
  Using $e^{\frac 3 2 qt^{-1}} = 1 + \frac 3 2 q t^{-1} + \m O(t^{-2})$, we have:
  \begin{align}
    \mr v^n(\tau,\eta) &= \mr v^o\left(\tau,\eta - \frac 3 2 q e^{-3\tau/2}\right) + \frac 3 2 q e^{-\tau/2}\phi_0(\eta-m(\tau)) + \m O(e^{-\tau}) \\
                       &= \mr v^o(\tau,\eta) + \frac 3 2 q e^{-\tau/2}\phi_0(\eta-m(\tau)) + \m O(e^{-\tau})
  \end{align}
  in $L^2(\R_+)$.

  The approximate solution $V_{\op{app}}$ is changed through $V_3^+$ by $e^{-\tau/2}\left(\frac 3 2 q \phi_0 + q\phi_1\right)$.
  So
  \begin{equation}
    \mr V_{\op{app}}^n(\tau,\eta) = \mr V_{\op{app}}^o(\tau,\eta) + \frac 3 2 q e^{-\tau/2}\phi_0(\eta-m(\tau)) + qe^{-\tau/2}\phi_1(\eta-m(\tau)).
  \end{equation}

  Recall that on $[2m(\tau),\infty)$, $w = e^{\eta^2/8}(\mr v - \mr V_{\op{app}})$.
  Thus the above observations imply
  \begin{equation}
    w^n(\tau,\eta)-w^o(\tau,\eta) = -q e^{\eta^2/8}e^{-\tau/2}\phi_1(\eta-m(\tau)) + \m O(e^{-\tau}) = -re^{-\tau/2}e_1(\eta) + \m O(e^{(\eps-1)\tau}).
  \end{equation}
  Hence by Lemma \ref{lem:r},
  \begin{equation}
    \lim_{\tau\to\infty}e^{\tau/2}\abs{\braket{e_1,w^n}} = 0.
  \end{equation}

  For the remainder of the proof we use the new forms of all functions defined with $q_3=q$, and drop the superscript $n$.
  The calculations in the proof of Lemma \ref{lem:r} continue to hold for $w$, but now \eqref{eq:q_3_limit} implies
  \begin{equation}
    \braket{e_1,w} = -e^{-\tau/2}\int_\tau^\infty e^{\tau'/2}\nu_1(\tau')\;d\tau'
  \end{equation}
  with $\abs{\nu_1(\tau)}\lesssim e^{(4\eps-1)\tau}$, so
  \begin{equation}
    \abs{\braket{e_1,w}} \lesssim e^{(4\eps-1)\tau}.
  \end{equation}
  By \eqref{eq:w_perp_bd},
  \begin{equation}
    \norm{w}_{L^2(\R_+)} \lesssim e^{(2\eps-3/4)\tau}.
  \end{equation}
  We now wish to obtain uniform bounds on $w$ as well.

  Fix $A>0$ large enough that $\frac{\eta^2}{16}-\frac 3 4-100\eta-100\geq 0$ for $\eta\geq A$.
  On the interval $[0,A]$, parabolic regularity implies
  \begin{equation}
    \norm{w}_{L^\infty[0,A]}\leq C e^{(2\eps-3/4)\tau}
  \end{equation}
  for $\tau\geq 1$.
  Now consider a maximum of $\abs w$ on $[A,\infty)$.
  There $w_\eta$ vanishes, so our previous bounds imply
  \begin{equation}
    w_\tau + \left[\m M + \frac 1 4 g(\tau)\eta - h(\tau)\right]w = E
  \end{equation}
  with $\norm{E}_{L^\infty(\R_+)}\lesssim e^{(4\eps-1)\tau}$ and $\abs{g},\abs{h}\leq 100$.
  By \eqref{eq:w}, the form of $\m M$, and the definition of $A$, any maximum of $\abs{w}$ on $[A,\infty)$ larger than $Ce^{(4\eps-1)\tau}$ will decrease in magnitude as $e^{-3\tau/4}$.
  Since $w$ is initially bounded, this implies $\norm{w}_{L^\infty[A,\infty)} \leq Ce^{(2\eps-3/4)\tau}$.
  Combining this with the bound on $[0,A]$, we obtain
  \begin{equation}
    \label{eq:mid}
    \norm{w}_{L^2(\R_+)} + \norm{w}_{L^\infty(\R_+)} \lesssim e^{(2\eps-3/4)\tau}\quad\text{for } \tau\geq 1.
  \end{equation}

  Next, we wish the use the Dirichlet condition on $w$ to show that in fact
  \begin{equation}
    \abs{w(\tau,\eta)}\lesssim \eta e^{(2\eps-3/4)\tau}\quad\text{for }\tau\geq 1.
  \end{equation}
  By the Kato inequality, on a sufficiently small interval $\eta\in (0,a)$ with $a>0$,
  \begin{equation}
    \partial_\tau \abs{w} - \partial_{\eta\eta}\abs{w} - 10 \abs w - g(\tau)\partial_\eta\abs{w} \leq C\left[e^{(4\eps-1/2)\tau}\tbf{1}_{[0,3m(\tau)]}(\eta) + e^{(4\eps-1)\tau}\right].
  \end{equation}
  By \eqref{eq:mid}, we have boundary conditions $\abs{w}(\tau,0)=0$ and $\abs{w}(\tau,a) \leq C e^{(2\eps-3/4)\tau}$.
  Let $\varphi_0$ solve $-\partial_{\eta\eta}\varphi_0 = Ce^{(4\eps-1/2)\tau}\tbf{1}_{[0,3m(\tau)]}$ on $(0,a)$ with $\varphi_0(0)=\varphi_0(a)=0.$
  Then $\varphi_0$ is explicitly given by:
  \begin{equation}
    \varphi_0(\eta) = 
    \begin{cases}
      \frac{C}{2}e^{(4\eps-1/2)\tau} \eta\left(6m(\tau)-\frac{9m(\tau)^2}{a}-\eta\right) & \text{for }\eta\in [0,3m(\tau)]\\
      C e^{(4\eps-1/2)\tau}\frac{9m(\tau)^2}{2a}(a-\eta) & \text{for }\eta\in [3m(\tau),a].
    \end{cases}
  \end{equation}
  From this form, we see that $\varphi_0(0)\leq Ce^{(5\eps-1)\tau}\eta$.
  We may then write $\abs{w} \leq \varphi_0+Ce^{(2\eps-3/4)\tau}\varphi_1$ with $\varphi_1$ satisfying
  \begin{equation}
    \partial_\tau\varphi_1 - \partial_{\eta\eta}\varphi_1 - 11\varphi_1 - g(\tau)\partial_\eta\varphi_1 = e^{(2\eps-1/4)\tau},\quad \varphi_1(\tau,0)=0,\; \varphi_1(\tau,a)=1.
  \end{equation}
  By choosing $a$ small, we may ensure the eigenvalue $\lambda_a$ of the Dirichlet Laplacian on $(0,a)$ satisfies $\lambda_a>100$.
  This forces $\varphi_1 \leq C\eta$.
  Therefore $\abs{w}(\tau,\eta)\lesssim \eta e^{(2\eps-3/4)\tau}$ when $\tau\geq 1$, as desired.

  In summary, we have bootstrapped \eqref{eq:w_bound0} to
  \begin{equation}
    \label{eq:w_mid}
    \norm{w}_{L^2(\R_+)} + \norm{w}_{L^\infty(\R_+)} \lesssim e^{(\eps-3/4)\tau},\quad \abs{w(\tau,\eta)}\lesssim \eta e^{(2\eps-3/4)\tau}\quad \text{for }\tau\geq 1.
  \end{equation}
  However, this bound is still weaker than \eqref{eq:w_bound1}.
  We improve it further by performing the computations in the proof of Lemma \ref{lem:r} again, now using \eqref{eq:w_mid} and
  \begin{equation}
    \abs{\braket{e_0,w}} + \abs{\braket{e_1,w}} \lesssim e^{(4\eps-1)\tau}.
  \end{equation}
  The term of concern is thus $\|w^\perp\|$.

  Consider \eqref{eq:w_perp}.
  We wish to control $\|w^\perp\|$ with error $\m O(e^{(C\eps - 2)\tau})$.
  Hence our earlier bounds on $E_1$ and $E_2$ suffice.
  Now note that \eqref{eq:w_mid} holds with $w$ replaced by $w^\perp$.
  So
  \begin{align}
    \abs{e^\tau \int_{\R_+}\exp\left([m(\tau)-\eta] e^{\tau/2}\right)(\mr v + \mr V_{\op{app}})w w^\perp\;d\eta} &\lesssim e^{(3\eps-1/2)\tau}\int_{\R_+}\eta^2(1+\eta^3 e^{3\tau/2})\exp\left(- \eta e^{\tau/2}\right)\;d\eta \\
                                                                                                                 &\lesssim e^{(3\eps-2)\tau}.
  \end{align}

  Following the earlier analysis of the $E_3$ term, we find
  \begin{equation}
    \int_{\R_+}w_\eta w^\perp \lesssim \left(\abs{\braket{e_0,w}}+\abs{\braket{e_1,w}}\right)\norm{w} \lesssim e^{(6\eps-7/4)\tau}
  \end{equation}
  and
  \begin{equation}
    \int_{\R_+}\eta ww^\perp \leq 16\Braket{\m M w^\perp,w^\perp} + 2\|w^\perp\|^2 + C_\eps e^{(4\eps-3/2)\tau}.
  \end{equation}
  Thus
  \begin{equation}
    \braket{E_3,w^\perp} \leq Ce^{-\tau/2}\braket{Mw^\perp,w^\perp} + C_\eps e^{(5\eps-2)\tau}.
  \end{equation}
  Arguing as before, these bounds and \eqref{eq:w_perp} imply
  \begin{equation}
    \|w^\perp\|^2 \lesssim e^{(10\eps-2)\tau}.
  \end{equation}
  Therefore
  \begin{equation}
    \norm{w}_{L^2(\R_+)} \lesssim e^{(5\eps-1)\tau}.
  \end{equation}
  Repeating the $L^\infty$ arguments with this new control, we obtain \eqref{eq:w_bound1}.
  In particular,
  \begin{equation}
    \abs{w(\tau,\eta)} \lesssim \eta e^{(5\eps - 1)\tau}\quad  \text{for }\tau\geq 1.
  \end{equation}
  
  Finally, choose $\eps = \frac\gamma 6$.
  In the original variables, we find
  \begin{equation}
    \abs{v(t,x)-V_{\op{app}}(t,x)} \leq C_\gamma\left(\frac{x + t^\eps}{\sqrt t} \right)t^{5\eps - 1} \leq \frac{C_\gamma(1+\abs x)}{t^{\frac 3 2 - \gamma}}
  \end{equation}
  when $t\geq 3$ and $x\geq 2 - t^\eps$.
  This concludes the proof of Theorem \ref{thm:v}.
\end{proof}

\section{Appendix}
\label{sec:appendix}

In this appendix, we use standard ODE theory to establish Lemma \ref{lem:ODE}.

\begin{proof}[Proof of Lemma \textup{\ref{lem:ODE}}.]
  We first show that there exists a solution to \eqref{eq:V1-} decaying as $x\to-\infty$.
  Consider the traveling front $\phi$, which satisfies
  \begin{equation}
    \label{eq:phi}
    \phi'' + 2\phi' + \phi - \phi^2 = 0.
  \end{equation}
  Expanding this equation around $\phi=1$, we find that
  \begin{equation}
    \phi(x) = 1 - Ae^{(\sqrt 2-1) x} + \m O(e^{2(\sqrt{2}-1)x})\quad \text{as }x\to-\infty
  \end{equation}
  for some $A>0$.
  Recalling that $V_0^-(x) = e^x\phi(x)$, we have:
  \begin{equation}
    \frac 3 2 [V_0^--(V_0^-)'] = \frac 3 2 A(\sqrt 2 - 1)e^{\sqrt 2 x} + \m O(e^{(2\sqrt 2 - 1)x}).
  \end{equation}
  So \eqref{eq:V1-} has the form
  \begin{equation}
    \label{eq:simple}
    -V'' + 2V = FV + \frac 3 2 A(\sqrt 2 - 1)e^{\sqrt 2 x} + G,
  \end{equation}
  where $F = \m O(e^{(\sqrt 2 - 1)x})$ and $G = \m O(e^{(2\sqrt 2 - 1)x})$.
  We construct a series solution to \eqref{eq:simple}.
  We first seek a decaying solution to
  \begin{equation}
    \label{eq:y0}
    -V_0'' + 2V_0 = \frac 3 2 A(\sqrt 2 - 1)e^{\sqrt 2 x} + G.
  \end{equation}
  The homogeneous solutions to $-V'' + 2V = 0$ are $e^{\pm \sqrt 2 x}$.
  Thus by the theory of constant-coefficient ODEs, there exists a solution to \eqref{eq:y0} of the form
  \begin{equation}
    V_0 = -\frac 3{4\sqrt{2}}A(\sqrt 2 - 1)xe^{\sqrt 2 x} + \m O\left(e^{(2\sqrt 2  - 1)x}\right).
  \end{equation}
  Thus for fixed small $\delta>0$, there exists $C>0$ such that
  \begin{equation}
    \abs{V_0(x)}\leq C e^{(\sqrt 2 -\delta)x}\quad \text{for }x\leq 0.
  \end{equation}
  Choose $C$ large enough that $\abs{F(x)}\leq Ce^{(\sqrt 2 - 1)x}$.
  Then define a sequence of functions $(V_k)$ by
  \begin{equation}
    -V_{k+1}'' + 2V_{k+1} = FV_k, \quad \lim_{x\to-\infty}e^{\sqrt 2 x}V_{k+1}(x)=0,\quad \text{for }k\in \Z_{\geq 0}.
  \end{equation}
  We will show by induction that
  \begin{equation}
    \label{eq:induct}
    \abs{V_k(x)}\leq \frac{C^{k+1}}{(\sqrt 2 - 1 - \delta)^k k!}e^{[k(\sqrt 2 -1)+\sqrt 2 -\delta]x}.
  \end{equation}
  This already holds for $V_0$, so suppose it holds for $V_k$.
  We can bound $V_{k+1}$ by writing the second-order equation for $V_{k+1}$ as a first-order system, which we solve with matrix exponentials.
  Taking norms, we obtain:
  \begin{align}
    \abs{V_{k+1}}(x) &\leq \int_{-\infty}^x e^{\sqrt 2 (x-y)}\abs{FV_k(y)}\;dy \leq \int_0^\infty e^{\sqrt{2}z}\abs{FV_{k}(x-z)}\;dz\\
                     &\leq \frac{C^{k+2}}{(\sqrt 2 - 1 - \delta)^k k!} e^{[(k+1)(\sqrt 2 -1)+\sqrt 2 -\delta]x}\int_0^\infty e^{-[(k+1)(\sqrt 2 - 1) - \delta]z}\;dz.
  \end{align}
  Bounding the final integral, by $[(k+1)(\sqrt 2-1-\delta)]^{-1}$, we have \eqref{eq:induct}.
  Similar bounds can be shown for $V_{k+1}'$ and $V_{k+1}''$.
  Thus
  \begin{equation}
    V \coloneqq \sum_{k\geq 0}V_k
  \end{equation}
  converges in $C^2(\R)$, and solves
  \begin{equation}
    -V'' + 2V = -V_0'' + 2V_0 + \sum_{k\geq 0}(-V_{k+1}'' + 2V_{k+1}) = FV + \frac 3 2 A(\sqrt 2 - 1)e^{\sqrt 2 x} + G.
  \end{equation}
  Finally, $V = \m O(e^{(\sqrt 2 - \delta)x})$ on $\R_-$, so $V$ is a decaying solution to \eqref{eq:simple}, as desired.

  Now let
  \begin{equation}
    \mr V(x) \coloneqq e^x\phi'(x).
  \end{equation}
  Equation \eqref{eq:phi} implies
  \begin{equation}
    \label{eq:V_homog}
    -\mr V'' + 2V_0^-\mr V = 0.
  \end{equation}
  With the bounds noted previously, we have
  \begin{align}
    \mr V(x) &= \m O(e^{\sqrt 2 x})\quad\text{as } x\to-\infty,\\
    \mr V(x) &= 1 - x + \m O(e^{-\omega x}) \quad \text{as }x\to +\infty.
  \end{align}
  So $\mr V$ is a solution of the homogeneous equation \eqref{eq:V_homog} which decays at $-\infty$ and has known asymptotics at $+\infty$.
  
  Now consider the behavior of $V$ as $x\to\infty$.
  We claim that $V$ satisfies
  \begin{equation}
    V(x) = -\frac 1 4 x^3 + \frac 3 4 x^2 + C_1x + C_0 + \m O(e^{-\omega x/2})
  \end{equation}
  for some $C_1,C_0\in \R$.
  Let
  \begin{equation}
    Z(x)\coloneqq V(x) + \frac 1 4 x^3 - \frac 3 4 x^2.
  \end{equation}
  Then since $V_0^-(x) = x - 1 + \m O(e^{-\omega x})$ as $x\to\infty$, we have
  \begin{equation}
    \label{eq:z}
    Z'' = HZ + K
  \end{equation}
  with $H,K = \m O(e^{-\omega x})$.
  We will argue that $Z = C_1x + C_0 + \m O(e^{-\omega x/2})$.
  Fix $B\geq 0$ such that $\abs{H(x)}\leq \frac {\omega^2}{4}$ when $x\geq B$.
  We solve \eqref{eq:z} using the matrix exponential again.
  Taking norms, we can show that $\abs{Z}$ is dominated on $[B,\infty)$ by solutions to $\ti Z'' = \frac {\omega^2}{4}\ti Z$, namely linear combinations of $e^{\pm \omega x/2}$.
  So $\abs{Z(x)} \leq C e^{\omega x/2}$ on $[B,\infty)$.
  With this \emph{a priori} bound, we see that $Z'' = \m O(e^{-\omega x/2})$ on $[B,\infty)$.
  Integrating twice, we obtain
  \begin{equation}
    Z(x) = C_1x + C_0 + \m O(e^{-\omega x/2})
  \end{equation}
  for some $C_1,C_0\in \R$, as desired.

  Finally, let
  \begin{equation}
    V_1^-\coloneqq V - C_0\mr V.
  \end{equation}
  Then $V_1^-\in C(\R)$ solves \eqref{eq:V1-}, and satisfies the bounds in Lemma \ref{lem:ODE} with $C_1^- = C_1 + C_0$.
  Although we have not explicitly discussed $(V_1^-)'$, its bounds follow just as those for $V_1^-$.
\end{proof}

\bibliographystyle{amsplain}
\bibliography{KPP}

\providecommand{\bysame}{\leavevmode\hbox to3em{\hrulefill}\thinspace}
\providecommand{\MR}{\relax\ifhmode\unskip\space\fi MR }
\providecommand{\MRhref}[2]{%
  \href{http://www.ams.org/mathscinet-getitem?mr=#1}{#2}
}
\providecommand{\href}[2]{#2}
\begin{thebibliography}{10}

\bibitem{BBD}
J.~{Berestycki}, {\'E}.~{Brunet}, and B.~{Derrida}, \emph{{Exact solution and
  precise asymptotics of a Fisher-KPP type front}}, arXiv e-prints (2017).

\bibitem{Bramson78}
M.~Bramson, \emph{Maximal displacement of branching {B}rownian motion}, Comm.
  Pure Appl. Math. \textbf{31} (1978), no.~5, 531--581.

\bibitem{Bramson83}
\bysame, \emph{Convergence of solutions of the {K}olmogorov equation to
  travelling waves}, Mem. Amer. Math. Soc. \textbf{44} (1983), no.~285.

\bibitem{EvS}
U.~Ebert and W.~van Saarloos, \emph{Front propagation into unstable states:
  universal algebraic convergence towards uniformly translating pulled fronts},
  Phys. D \textbf{146} (2000), no.~1, 1--99.

\bibitem{Fisher}
R.~A. Fisher, \emph{The wave of advance of advantageous genes}, Ann. Eugen.
  \textbf{7} (1937), no.~4, 355--369.

\bibitem{HNRR}
F.~Hamel, J.~Nolen, J.-M. Roquejoffre, and L.~Ryzhik, \emph{A short proof of
  the logarithmic {B}ramson correction in {F}isher-{KPP} equations}, Netw.
  Heterog. Media \textbf{8} (2013), no.~1, 275--289.

\bibitem{Henderson}
C.~{Henderson}, \emph{Population stabilization in branching {B}rownian motion
  with absorption}, arXiv e-prints (2014).

\bibitem{KPP}
A.~N. Kolmogorov, I.~G. Petrovsky, and N.~S. Piskunov, \emph{{\'E}tude de
  l{'}{\'e}quation de la diffusion avec croissance de la quantit{\'e} de
  mati{\`e}re et son application {\`a} un probl{\`e}me biologique}, Moscow
  Univ. Math. Bull. \textbf{1} (1937), 1--26.

\bibitem{Lau}
K.-S. Lau, \emph{On the nonlinear diffusion equation of {K}olmogorov,
  {P}etrovsky, and {P}iscounov}, J. Differential Equations \textbf{59} (1985),
  no.~1, 44--70.

\bibitem{Lehmer}
D.~H. Lehmer, \emph{Interesting series involving the central binomial
  coefficient}, Am. Math. Mon. \textbf{92} (1985), no.~7, 449--457.

\bibitem{NRR1}
J.~{Nolen}, J.-M. {Roquejoffre}, and L.~{Ryzhik}, \emph{{Convergence to a
  single wave in the Fisher-KPP equation}}, arXiv e-prints (2016).

\bibitem{NRR2}
\bysame, \emph{{Refined long time asymptotics for Fisher-KPP fronts}}, arXiv
  e-prints (2016).

\end{thebibliography}

\end{document}